\documentclass[11pt]{amsart}
\usepackage{amssymb}
\usepackage{amsmath}
\usepackage{amscd}
\usepackage{curves}
\usepackage{epsfig}
\usepackage{bbm}
\usepackage{geometry}
\usepackage{graphicx}
\usepackage{pstricks}
\usepackage{pstricks-add}
\usepackage{pst-grad}
\usepackage{pst-plot}
\usepackage{verbatim}
\usepackage{latexsym}
\usepackage{eucal}
\usepackage{amsfonts}
\usepackage{enumerate}
\numberwithin{equation}{section}
\newtheorem{theorem}{Theorem}[section]
\newtheorem{lemma}[theorem]{Lemma}
\newtheorem{prop}[theorem]{Proposition}

\newtheorem{definition}[theorem]{Definition}

\newtheorem{remark}[theorem]{Remark}
\def \bpf {\begin{proof}}
\def \epf {\end{proof}}
\def \beq {\begin{equation*}}
\def \eeq {\end{equation*}}
\def \bsp{\begin{split}}
\def \esp{\end{split}}

\def \geqs {\gtrsim}
\def \leqs {\lesssim}
\def \ido {{\stackrel{o}{\operatorname{I}}}}

\def \CI {{C^\infty}}

\def \wt {\widetilde}

\def \mca {{\mathcal A}}

\def \mce {{\mathcal E}}
\def \mcf {{\mathcal F}}
\def \mcg {{\mathcal G}}
\def \mch {{\mathcal H}}
\def \mci {{\mathcal I}}

\def \mcl {{\mathcal L}}
\def \mcm {{\mathcal M}}
\def \mcw {{\mathcal W}}
\def \mco {{\mathcal O}}
\def \mcp {{\mathcal P}}
\def \mcq {{\mathcal Q}}
\def \mcr {{\mathcal R}}

\def \mcu {{\mathcal U}}
\def \mcv {{\mathcal V}}

\def \zed {{\mathcal Z}}

\def \mbr {{\mathbb R}}
\def \mr {{\mathbb R}}
\def \mn {{\mathbb N}}

\def\ha {\frac{1}{2}}
\def\tha {\frac{3}{2}}
\def\fha {\frac{5}{2}}
\def \oq {\frac{1}{4}}

\def \ka {\kappa}

\def \Id {\operatorname{Id}}
\def \comp {\operatorname{comp}}

\def \loc {{\operatorname{loc}}}

\def \diag{\operatorname{Diag}}

\def \supp {\text{supp }}

\def \mcn {\mathcal{N}}

\def \mchz {{\mathcal{H}^{0,m}}(\Omega_2)}
\def \mcho {{\mathcal{H}^{1,m}}(\Omega_2)}

\def \eps {\varepsilon}   
\def \vphi {\varphi}   
   
\def \La {\Lambda}   
\def \lan {\langle}   
\def \ran {\rangle}   
\def \del {\delta}   

\def \p {\partial}

\def \novt {\frac{n}{2}}

\def \beqq {\begin{equation}}
\def \eeqq {\end{equation}}

\def \mck {\mathcal{K}}

\numberwithin{equation}{section}

\begin{document}
\title[Singularities of Semilinear Waves]{ Singularities Generated by the Triple Interaction of Semilinear Conormal Waves}
\author{ Ant\^onio S\'a Barreto and Yiran Wang }
\address{Ant\^onio S\'a Barreto\newline
\indent Department of Mathematics, Purdue University \newline
\indent 150 North University Street, West Lafayette Indiana,  47907, USA}
\email{sabarre@purdue.edu}
\address{Yiran Wang\newline \indent Stanford University, Department of Mathematics \newline \indent  Building 380, Stanford, CA 94305}
\email{yrw@stanford.edu }
\begin{abstract}  We study the local propagation of conormal singularities for solutions of semilinear wave equations  $\square u= P(y,u),$  where $P(y,u)$ is a polynomial of degree $N\geq 3$ in $u$ with $C^\infty(\mr^3_y)$ coefficients.  We know from the work of Melrose \& Ritter  and Bony that if $u$ is conormal to three waves which intersect transversally at point $q,$ then after the triple interaction $u(y)$ is a conormal distribution with respect to the three waves and the characteristic cone  $\mcq$ with vertex at $q.$ We compute the principal symbol of $u$ at the cone and away from the hypersurfaces.  We show that if $\p_u^3P(q,u(q))\not=0,$   $u$ is an ellipitic  conormal distribution.
\end{abstract}

\maketitle
\section{Introduction}

We study the propagation of conormal singularities for solutions $u(y) \in H^s_{\loc}(\Omega),$ $\Omega\subset \mr^3,$ and $s>\frac{3}{2},$ of  semilinear wave equations of the form
\begin{equation}
\begin{split}
& \square  u = P(y,u) = \sum_{j=0}^N \zed(y) a_j(y) u^j, \ N \in \mn,  N\geq 3, \  a_j \in C^\infty(\Omega), \\
  & u(y)= v(y), \;\ t<-1,
\end{split} \label{Weq}
\end{equation}
where $\square$ is a second order strictly hyperbolic operator,  $t$ is a time function for $\square,$ $\Omega$ is a relatively compact neighborhood of $0\in \mr^3$ which is bicharacteristically convex with respect to $\square$ and $\zed(y)$ is $C^\infty$ and compactly supported, and  $\zed=0$ for $t<-1.$   The existence and uniqueness of the solution $u$ in this range of Sobolev regularity is well known for $\Omega$ small enough and we want to analyze the singularities of $u.$

We shall assume that the initial data $v(y)= v_1(y)+ v_2(y)+ v_3(y),$ where $v_j(y)$ is a conormal distribution to a $C^\infty$ hypersurface $\Sigma_j\subset \Omega $ which is closed and characteristic for $\square.$ Moreover, we shall assume that $\Sigma_1\cap\Sigma_2\cap \Sigma_3=\{0\}$ and the normal vectors $\mcn_j$ to $\Sigma_j,$ $j=1,2,3,$ are linearly independent at $\{0\}.$ We also assume that $0\in \{t=0\}$  and $\zed(y)=1$ near $\{0\}.$      The conormality assumption is fundamental; M. Beals \cite{Beals1} showed that without this assumption singularities may self-spread and the singular support of $u$ can propagate in the same way as its support.

The study of the propagation of singularities for nonlinear wave equation started in the late 1970's with the work of Bony \cite{Bony1,Bony2}. Bony  \cite{Bony3,Bony4} also started the study of the interaction of nonlinear conormal waves  in the early 1980s and was followed by many people including M. Beals \cite{Beals1,Beals4}, Chemin \cite{Chemin}, Delort \cite{Delort1,Delort2}, Joshi and S\'a Barreto \cite{JosSab}, Lebeau \cite{Lebeau1}, Melrose and Ritter \cite{MelRit}, Melrose and S\'a Barreto \cite{MelSab}, Melrose, S\'a Barreto and Zworski \cite{MelSabZwo},  Piriou \cite{Pir}, Rauch and Reed \cite{RauRee0,RauRee01,RauRee}, S\'a Barreto \cite{SaB,SaB1} and  Zworski \cite{Zworski1}.

Our  renewed interest in the topic comes from recent applications to inverse problems for semilinear wave equations in the work of Kurylev, Lassas and Uhlmann \cite{KLU}, Lassas, Uhlmann and Wang \cite{LUW}, and Uhlmann and Wang \cite{UW}.

If $N=1,$ and hence equation \eqref{Weq} is linear, the superposition principle holds and  $u(y)= u_1(y)+u_2(y) + u_3(y),$ where $u_j$ is a conormal distribution to $\Sigma_j.$ 
  Bony \cite{Bony3,Bony4}  and later Melrose and Ritter \cite{MelRit} proved that for arbitrary $P\in C^\infty(\Omega\times \mathbb{R})$,  if $v_2=v_3=0,$  and $v_1$ is conormal to $\Sigma_1,$ the solution $u(y)$ remains conormal to $\Sigma_1$  in $\Omega.$ Similarly, if $v_3=0,$ and $v_j$ is conormal to $\Sigma_j,$ $j=1,2,$  then $u$ remains conormal to $\Sigma_1\cup \Sigma_2$ in $\Omega.$   However in the case of three waves, this is no longer true.  One of the first examples of the appearance of new singularities in the interaction of three waves,  in the case of a system  $\square u_j=0,$  $j=1,2,3$ and $\square u= u_1u_2u_3,$ was given  by  Rauch and Reed \cite{RauRee}.  They showed that $u$ has additional singularities  on $\mcq,$  the characteristic cone for $\square$  with vertex at $\{0\}.$

Melrose and Ritter \cite{MelRit} and Bony \cite{Bony5,Bony6}, independently and using very different methods, showed that if $v_j$ is conormal to $\Sigma_j,$ $j=1,2,3,$ then, for  arbitrary $P(y,u)\in C^\infty,$ the solution $u(y)$  to \eqref{Weq} is conormal to $\Sigma_1\cup\Sigma_2\cup\Sigma_3 \cup \mcq.$  This shows that the only possible additional singularities resulting from the interaction of three transversal conormal waves are contained in $\mcq,$ but $u$ could be smooth there.    The novelty of this paper is the computation of the principal symbol of $u$ on the cone, and away from the incoming hypersurfaces.

In the next section we will define the space $I^m(\Omega,\Sigma)$ of conormal distributions to a $\CI$ closed submanifold $\Sigma.$   
We first state a version of our main result leaving out most  technical details and we refer the reader to Theorem \ref{triple1}   for the precise statement.

 \begin{theorem}\label{triple}  Let $\Omega,$ $\square,$ $P(y,u),$ $\Sigma_j,$ $j=1,2,3$ be as above and  let $\mcq$ be the characteristic cone  for $\square$ with vertex at $0.$  Let $v_j\in I^{m-\oq}(\Omega,\Sigma_j),$ $m<-\fha,$  $j=1,2,3$ be elliptic conormal distributions and let $u(y)\in H^{-m-\ha-}(\Omega),$ be the corresponding solution to \eqref{Weq}.   Then there exists a neighborhood $\mco\subset \Omega$ of $\{0\}$ such that for any open subset $\widetilde{\mco}\subset \mco,$ such that $\widetilde{\mco}\cap \Sigma_j=\emptyset,$ $j=1,2,3,$ and $\mco\cap\mcq\not=\emptyset,$ 
 $u \in I^{3m-\frac34}(\widetilde{\mco}, \mcq).$ If  $(\p_u^3 P)(0,u(0))\not=0,$ then away from $\Sigma_j,$ $j=1,2,3,$  $u$ is an elliptic conormal distribution to $\mcq.$    If $(\p_u^3 P)(0,u(0))=0,$ the possible singularities on $\mcq$ are of lower order and $u$ may be $\CI$ there.
 \end{theorem}
 
 The proof of Theorem \ref{triple} relies on the results of Melrose and Ritter \cite{MelRit} and Bony \cite{Bony4,Bony5,Bony6}, but it gives the order of the singularity of $u$ at the cone. 
Theorem \ref{triple1} below actually gives the principal symbol of $u,$ and shows that  one can recover $(\p_u^3 P)(0,u(0))$ from the leading singularity of the solution $u$ to \eqref{Weq} on the cone $\mcq.$  
The particular case of Theorem \ref{triple} in which $P(y,u)=\zed(y) y^3,$ and $\square$ has constant coefficients, and the incoming waves are classical conormal,  is due to M.Beals \cite{Beals4}.

  Proposition \ref{asyFT}  below allows us to adapt M. Beals'  methods to prove Theorem \ref{triple1}. These spaces  make it possible to avoid additional technical difficulties involving symbol expansions and propagation of product type conormal distributions, see for example the work of Eswarathasan \cite{Suresh}, Greenleaf and Uhlmann \cite{GreUhl},   Joshi \cite{Joshi}, Melrose and Uhlmann \cite{MelUhl} and references cited there.


  It is also important to emphasize that we assume the hypersurfaces $\Sigma_j$ remain smooth throughout $\Omega$ and no caustics are formed.    The propagation of singularities for solutions of semilinear wave equations when caustics develop has been studied by several people including  M. Beals \cite{Beals2},   Delort \cite{Delort1,Delort2},  Joshi and S\'a Barreto \cite{JosSab}, Lebeau \cite{Lebeau1},  Melrose \cite{Melrose1,Melrose2},  Melrose and S\'a Barreto  \cite{MelSab},  S\'a Barreto \cite{SaB1} and Zworski \cite{Zworski1}.

 \section{Spaces of Distributions}

We recall the definition of  some spaces of  distributions.   Throughout the paper we will use both $\mcf(\vphi)$ and $\widehat{\vphi}$ to denote the Fourier transform of $\vphi.$
  As usual, $H^s(\mr^{n}),$ $s\in \mr,$  denotes the Sobolev spaces. The definition of the Besov spaces ${}^p H_s(\mr^{n}),$ $1\leq p \leq \infty,$ can be found in Appendix B.1 of H\"ormander's book \cite{HormanderV3}. We shall say that   $u \in H^s_{\loc}(\mr^{n})$ or $u\in {}^p H_s^{\loc}(\mr^{n}),$   
   if $\chi u \in H^s(\mr^{n})$  or $\chi u \in {}^p H_s(\mr^{n})$ for every $\chi\in C_0^\infty(\mr^{n}).$
 We shall say that $u\in H^{s-}(\mr^{n})$ if $u\in H^{s-\eps}(\mr^{n})$ for all $\eps>0.$

\subsection{Conormal Distributions}   Let  $\Omega \subset \mr^{n}$  denote an open and relatively compact subset.  Even though our main results are for $n=3,$ we will not make this restriction in this and in the next subsection.  Following  H\"ormander \cite{HormanderV3},  $u$ is said to be a conormal distribution of order $m$ with respect to a submanifold 
$\Sigma \subset \Omega$ of codimension $k$ and we denote $u\in I^{m}(\Omega, \Sigma)$  if for any $N\in \mn,$
 \begin{gather*}
V_1 V_2 \ldots V_N u \in {}^\infty H_{-m-\frac{n}{4}}^{\loc} (\Omega),
\end{gather*}
where $V_j,$ $1\leq j \leq N,$ are $C^\infty$ vector fields tangent to $\Sigma.$    

 According to Theorem 18.2.8 of \cite{HormanderV3}, $u\in I^{m}(\Omega, \Sigma)$ if and only if  $u \in C^\infty(\mr^{n}\setminus \Sigma)$ and near any point $p\in \Sigma$ and in local coordinates where 
  $\Sigma=\{y_1=y_2=\ldots= y_k=0\},$  $y=(y', y''),$  $y'=(y_1, y_2, \ldots, y_k),$ $y''\in \mr^{n-k},$
\begin{gather}
u(y)= \int_{\mr^k} e^{i  y'\cdot \eta' } a(\eta',y) \; d\eta', \;\ a\in S^{m+\frac{n-2k}{4}}( \mr^{k}_{\eta'} \times \mr^{n-k}_{y''}), \label{con1}
\end{gather}
where for $r\in \mr,$ $S^{r}(\mr^k \times \mr^{n})$ is the class of symbols satisfying
\begin{gather*}
|\p_{y''}^\alpha \p_{\eta'}^\beta a(\eta',y'')| \leq C_{\alpha, \beta}(1+|\eta'|)^{r-|\beta|}.
\end{gather*}
These symbol spaces satisfy
\begin{gather*}
S^r(\mr^k\times \mr^{n-k}) \subset S^{r'}(\mr^k\times \mr^{n-k}), \;\ r<r',
\end{gather*}
and the space of distributions satisfy
\begin{gather*}
I^{m}(\Omega, \Sigma) \subset I^{m'}(\Omega, \Sigma), \;\ m<m'.
\end{gather*}

We will need the following result, which is Proposition 18.2.3 of \cite{HormanderV3}:
\begin{prop}\label{reduction}  Let $y=(y',y''),$ $y'=(y_1,\ldots,y_k)$  and let $\Sigma=\{y'=0\}$ and let $u \in I^m(\Omega,\Sigma).$  If 
$\alpha\in \mn^k$ then ${y'}^\alpha u \in I^{m-|\alpha|}(\Omega,\Sigma).$
\end{prop}
  If $u$ satisfies \eqref{con1}, and $y=(y',y''),$ $y'=(y_1,\ldots,  y_k),$ then by Taylor's formula
\begin{gather*}
a(\eta', y',y'')- \sum_{|\alpha|\leq k} \frac{1}{\alpha!} {y'}^\alpha \p_{y'}^\alpha a(\eta',0, y'')= O(|y'|^{k+1}),
\end{gather*}
and therefore,
\begin{gather}
\begin{gathered}
u(y)= \int_{\mr^k} e^{i y'\cdot \eta'} b(\eta',y'') d\eta'+ \mce, \text{ where } \mce \in \CI, \text{ and } \\
b(\eta',y'')\sim \sum_{\alpha} \frac{i^{|\alpha|}}{\alpha!} \p_{y'}^\alpha \p_{\eta'}^\alpha a(\eta',0,y'').
\end{gathered}\label{left-red}
\end{gather}
The principal symbol of $u$ is defined to be  the equivalence class of $a(\eta',0,y'')$ in the quotient $S^{m+\frac{n-2k}{4}}(\mr^k\times \mr^{n-k})/ S^{m+\frac{n-2k}{4}-1}(\mr^k \times \mr^{n-k})$ and the map
\begin{gather*}
 I^{m}(\Omega, \Sigma)/ I^{m-1}(\Omega, \Sigma) \longrightarrow S^{m+\frac{n-2k}{4}}(\mr^k\times \mr^{n-k})/ S^{m+\frac{n-2k}{4}-1}(\mr^k \times \mr^{n-k})\\
 [u] \longmapsto [a],
 \end{gather*}
is an isomorphism. The symbol map can be invariantly defined as in \cite{HormanderV3}, but since our analysis is completely local, we will not concern ourselves with that.

We will need  multiplicative properties of elements of $I^{m-\frac{n}{4}+\ha}(\Omega,\Sigma),$ see for example \cite{Pir}: 
\begin{prop}\label{alg-prop1} If $\Sigma\subset \Omega$ is a $C^\infty$ hypersurface, if $u, v\in I^{m-\frac{n}{4}+\ha}(\Omega,\Sigma)$ and $m<-1,$ then $uv\in I^{m-\frac{n}{4}+\ha}(\Omega,\Sigma).$
\end{prop}

\subsection{Piriou Spaces}   If $\Sigma\subset \Omega$ is a $C^\infty$ hypersurface and $u\in I^{m-\frac{n}{4}+\ha}(\Omega,\Sigma)$ is given by
\begin{gather}
u(y)= \int_{\mr} e^{i y_1\eta_1} a(\eta_1,y'') d\eta_1,  \ a\in S^m(\mr\times \mr^{n-1}), \ m<-1, \label{eq-pir1}
\end{gather}
the integral \eqref{eq-pir1} converges and $u$ is a continuous function.  In fact $\p_{y_1}^k u$ is continuous, provided $m+k<-1.$  We want to make sense of a finite order power series expansion for $u$ at $\Sigma$ for $m<-1.$  Let $k(m)$ be the non-negative integer such that $-m-2\leq k(m)< -m-1.$

Let $b\in S^{m+k(m)}$ be such that $a(\eta_1,y'')=\p_{\eta_1}^{k(m)} b(\eta_1,y'')$ for $|\eta_1|>1.$    One can construct such a function  by defining $b_1(\eta_1,y'')= \int_1^{\eta_1} a(t,y'') dt,$ if $\eta_1>1,$ and $b_1(\eta_1,y'')= \int_{\eta_1}^{-1} a(t,y'') dt,$ if $\eta_1<-1,$ 
$b_j(\eta_1,y'')= \int_1^{\eta_1} b_{j-1}(t,y'') dt,$ if $\eta_1>1,$ and $b_j(\eta_1,y'')= \int_{\eta_1}^{-1} b_{j-1}(t,y'') dt,$ if $\eta_1<-1.$ Then $b(\eta_1,y'')= b_{k(m)}(\eta_1,y'').$  

Finally, one can add a compactly supported $\CI$ function to $b$ and assume that $\int_\mr b(\eta_1,y'') d\eta_1=0.$ So we conclude that
if $u$ is given by \eqref{eq-pir1}, $a\in S^m(\mr\times \mr^{n-1}),$ $m<-1,$ then 
\begin{gather}
\begin{gathered}
u(y)= \mce(y)+ \int_\mr e^{iy_1\eta_1} \p_{\eta_1}^{k} b(\eta_1,y'') d\eta_1, \text{ where } b\in S^{m+k}, \  k=k(m), \  \mce\in C^\infty, \\ \text{ and }
\int_\mr b(\eta_1,y'') d\eta_1=0.
\end{gathered}
\end{gather}
If $k(m)=0,$ then 
\begin{gather*}
(u-\mce)(0,y'')= \int_\mr b(\eta_1,y'') d\eta_1=0.
\end{gather*}
If $k(m)\geq 1,$ since $m+k(m)<-1,$ $\p_{y_1}(u-\mce),$ the  oscillatory integral
\begin{gather*}
\p_{y_1}(u-\mce)(y_1,y'')= \int_\mr e^{i y_1\eta_1} i \eta_1\p_{\eta_1}^k b(\eta_1,y'') d\eta_1= \\
\int_\mr \p_{\eta_1}(e^{iy_1\eta_1} i \eta_1 \p_{\eta_1}^{k-1} b(\eta_1,y'')) d\eta_1- \int_\mr \p_{\eta_1}(e^{i y_1\eta_1} i \eta_1) \p_{\eta_1}^{k-1}b(\eta_1,y'') d\eta_1=\\ -\int_\mr \p_{\eta_1}(e^{iy_1\eta_1} i \eta_1) \p_{\eta_1}^{k-1}b(\eta_1,y'') d\eta_1
\end{gather*}
converges and
\begin{gather*}
\p_{y_1}(u-\mce)(0,y'')= -i\int_\mr  \p_{\eta_1}^{k-1}b(\eta_1,y'') d\eta_1=0.
\end{gather*}
Repeating this process inductively, we find that
\begin{gather*}
\p_{y_1}^j (u-\mce)(0,y'')=0, \ j\leq k(m).
\end{gather*}

We recall from  Piriou  \cite{Pir}:
\begin{definition}  If  $m<-1,$ let $k(m)$ be the non-negative integer such that $-m-2\leq k(m) < -m-1.$ If $\Sigma\subset \Omega$ is a $\CI$ hypersurface, we say that $u\in \ido{}^{m-\frac{n}{4}+\ha}(\Omega,\Sigma)$ if $u\in I^{m-\frac{n}{4}+\ha}(\Omega,\Sigma)$  and vanishes to order $k(m)+1$ at $\Sigma$  (all derivatives of $u$ to order less than or equal to $k(m)$ vanish at $\Sigma.$)
\end{definition}

From the discussion above we have the following:
\begin{prop}\label{pir-dec}   If $\Sigma\subset \Omega$ is a $\CI$ hypersurface, $u\in I^{m-\frac{n}{4}+\ha}(\Omega,\Sigma)$ and $m<-1,$ then $u=\mce+ v,$ with  $v \in \ido{}^{m-\frac{n}{4}+\ha}(\Omega,\Sigma)$ and $\mce\in C^\infty.$  If  $v \in \ido{}^{m-\frac{n}{4}+\ha}(\Omega,\Sigma)$ and $\Sigma=\{y_1=0\},$ then $v= y_1^{k(m)} w,$ $w \in I^{m+k(m)-\frac{n}{4}+\ha}(\Omega,\Sigma).$
\end{prop}

From this we deduce a useful multiplicative property of elements of $\ido{}^{m-\frac{n}{4}+\ha}(\Omega,\Sigma):$
\begin{prop}\label{pir-prod}   If $\Sigma\subset \Omega$ is a $\CI$ hypersurface,  $u \in \ido{}^{m-\frac{n}{4}+\ha}(\Omega,\Sigma)$ and $m<-1,$ then
\begin{gather}
u^j \in I^{m-(j-1)k(m)-\frac{n}{4}+\ha}(\Omega,\Sigma). \label{pir-prod1}
\end{gather}
\end{prop}
\begin{proof}  Suppose $\Sigma=\{y_1=0\}.$ Since  $u \in \ido{}^{m-\frac{n}{4}+\ha}(\Omega,\Sigma),$ then $u \in y_1^{k(m)}I^{m+k(m)-\frac{n}{4}+\ha}(\Omega,\Sigma).$  But since $m+k(m)<-1,$ it follows from Proposition \ref{alg-prop1} that 
\begin{gather*}
u^2 \in y_1^{2k(m)} I^{m+k(m)-\frac{n}{4}+\ha}(\Omega,\Sigma).
\end{gather*}
Then it follows from Proposition \ref{reduction} that
\begin{gather*}
u^2 \in I^{m-k(m)-\frac{n}{4}+\ha}(\Omega,\Sigma).
\end{gather*}
The general case follows by induction.
\end{proof}

Piriou \cite{Pir} establishes a much stronger result, namely: 
\begin{gather*}
\ido{}^{m_1-\frac{n}{4}+\ha}(\Omega,\Sigma)\cdot  \ido{}^{m_2-\frac{n}{4}+\ha}(\Omega,\Sigma)\subset \ido{}^{m_1+m_2+1-\frac{n}{4}+\ha}(\Omega,\Sigma), \\  \text{ provided } m_1<-1, m_2<-1, \ m_1, m_2 \not \in \mathbb{Z},
\end{gather*}
but \eqref{pir-prod1} is enough for our purposes.

\subsection{Spaces of Conormal Distributions Associated with Double and Triple Interaction} 
We introduce some $H^s$-based (versus Besov-space-based) conormal distributions that have been used in the study of propagation of conormal singularities. We also recall the results of Bony \cite{Bony3,Bony4,Bony5,Bony6} and Melrose and Ritter \cite{MelRit} about the propagation of conormal singularities for  semilinear wave equations.  

We follow  Melrose and Ritter  \cite{MelRit}, and for a Lie algebra and $\CI$ module of $\CI$ vector fields $\mcw$ we define the space of conormal distribtions with respect to $\mcw$ as
\begin{gather}
I H_{\loc}^s(\Omega, \mcw)=\{u \in H_{\loc}^s(\Omega):   V_1 V_2 \ldots V_N u \in H_{loc}^s (\Omega), \;\ V_j \in \mcw, \, N\in \mn_0\}.\label{con1N}
\end{gather}

In this paper we consider three such algebras:
\begin{gather}
\begin{gathered}
\mcv_j \text{ is the Lie algebra of } \CI \text{ vector fields tangent to } \Sigma_j, \\
\mcv_{jk}  \text{ is the Lie algebra of } \CI \text{ vector fields tangent to }  \Sigma_j \text{ and } \Sigma_k \\
\mcv_{123}  \text{ is the Lie algebra of } \CI \text{ vector fields tangent to } \Sigma_1, \Sigma_2 \text{ and } \Sigma_3 
\end{gathered}\label{TLA}
\end{gather}

One can show that, see for example \cite{MelRit},  these are finitely generated Lie algebras and $\CI$ modules and
\begin{gather}
\begin{gathered}
\mcv_1= C^\infty\text{span of } \{ y_1\p_{y_1},  \p_{y_2}, \p_{y_3}\}, \\
\mcv_{12}=C^\infty\text{span of } \{ y_1\p_{y_1},  \ y_2\p_{y_2}, \ \p_{y_3}\},\\
\mcv_{123}=C^\infty\text{span of } \{ y_1\p_{y_1},  \ y_2\p_{y_2}, \ y_3\p_{y_3}\}.
\end{gathered} \label{spanW}
\end{gather}

Similarly, we define $\mcv_\mcq$  to be the Lie algebra and $\CI$ module of $C^\infty$ vector fields tangent to $\mcq.$ This is also finitely generated, see \cite{MelRit}. 

We recall Bony's results \cite{Bony3,Bony4} on the propagation of conormal regularity with respect to one hypersurface and two transversal hypersurfaces, see also \cite{MelRit}:
\begin{theorem}\label{REGSD}  Let $u \in H_{\loc}^s(\Omega),$ $s>\novt,$ satisfy $\square u= f(y,u),$ $f\in \CI.$   If $\Sigma_j,$ $j=1,2,$ are closed $C^\infty$  characteristic hypersfurfaces intersecting transversally at 
$\Gamma\subset \{t=0\}.$
$u\in IH_{\loc}^s(\Omega,\mcv_{j})$ in $t<0,$ then   $u \in I H_{\loc}^{s}(\Omega, \mcv_{j}).$
 If $u\in IH_{\loc}^s(\Omega,\mcv_{jk})$ in $t<0,$ then  $u \in I H_{\loc}^{s}(\Omega, \mcv_{jk}).$
\end{theorem}

The following version of the propagation of conormality for three transversally intersecting waves is due to Bony \cite{Bony5,Bony6}, see also Chemin \cite{Chemin}: 
\begin{theorem}\label{REG}  Let $u \in H_{\loc}^s(\Omega),$ $s>\novt,$ satisfy $\square u= f(y,u),$ $f\in C^\infty.$  If $\Sigma_j,$ $j=1,2,3$ are closed $C^\infty$  characteristic hypersfurfaces intersecting transversally at $0\in \{t=0\}.$
 If $u\in IH_{\loc}^s(\Omega,\mcv_{123})$ in $t<0,$ then  for any $s'\in (\novt,s),$
\begin{gather*}
\begin{gathered}
u \in I H_{\loc}^{s'}(\Omega, \mcv_\mcq) \text{ away from } \Sigma_j, \;\ j=1,2,3,\text{ and }
u \in I H_{\loc}^{s'}(\Omega,\mcv_{123}) \text{ away from }  \mcq.
\end{gathered}\label{Bon-Ch}
\end{gather*}
\end{theorem}

The version of Melrose and Ritter \cite{MelRit} is slightly different and is based on $L^2$ instead of $H^s,$ but says more about the regularity of $u$ near the intersection of the cones and the hypersurfaces.

\subsection{Beals spaces}\label{sub-B} Next we recall properties of the spaces of  distributions introduced by M. Beals \cite{Beals4}  which he used to  prove Theorem \ref{triple} in the case  where $P(y,u)= \zed(y)u^3.$   In this subsection, we set $n=3$ since this is the case we are concerned with, and it simplifies some of the proofs.

\begin{definition}\label{BSP1}  (M. Beals \cite{Beals4})   For $k_1,k_2,k_3\in \mr_+$ and $ s\in \mr,$ we say that $u\in H_{\loc}^{s, k_1,k_2,k_3}(\Omega)$ if for any $ \vphi\in C_0^\infty(\Omega),$
\begin{gather} 
\lan \eta_1\ran^{k_1} \lan \eta_2\ran^{k_2} \lan \eta_3 \ran^{k_3} \lan \eta \ran^{s}  \widehat{\vphi u} \in L^2(\mr^3), \label{defBSP}
\end{gather}
where $\eta=(\eta_1,\eta_2,\eta_3),$ $\lan \eta_j\ran=(1+\eta_j^2)^\ha,$ and $\lan\eta\ran=(1+|\eta|^2)^\ha.$  We say that $u\in H_{\loc}^{s-, k_1,k_2,k_3}(\Omega)$ if $u\in H_{\loc}^{s-\eps, k_1,k_2,k_3}(\Omega)$ for all $\eps>0.$ 
\end{definition}

Notice that  microlocally in the region  where all variables $\eta_j$ are elliptic, that is when
\begin{gather}
\lan\eta_j \ran \geqs \lan \eta\ran, \; j=1,2,3, \label{elliptic}
\end{gather}
$H_{\loc}^{s, k_1,k_2,k_3}(\Omega)=H_{\loc}^{s+k_1+k_2+k_3}(\Omega).$  
We remark that  the spaces $H_{\loc}^{s, k_1,k_2,k_3}(\Omega)$ depend on the choice of coordinates $y.$   However, we can say more if the distributions are also conormal.  More precisely we have
\begin{prop}\label{coord-inv} Let $\Omega\subset \mbr^3$ and let $\Sigma_j,$ $j=1,2,3$ be closed $C^\infty$ hypersurfaces intersecting transversally at a point.  Let $\mcv_{123}$ denote the Lie algebra of $C^\infty$ vector fields tangent to 
$\Sigma_1,$ $\Sigma_2$ and $\Sigma_3,$ as in \eqref{TLA}.  Let $y=(y_1,y_2,y_3),$  be local coordinates 
in a neighborhood $U$ of $\{0\}$ such that $\Sigma_j=\{y_j=0\},$ $j=1,2,3.$   For $u \in IH^{s-}(\Omega,\mcv_{123}),$ $s\in \mr,$ and 
 $\ka_1,\ka_2, \ka_3\in \mn,$  $u\in H^{s-, \ka_1,\ka_2,\ka_3}(U)$ if and only if 
\begin{gather} 
V_1^{j_1} V_2^{j_2} V_3^{j_3} \lan D_y\ran^{s-\eps} {u} \in L^2(U), \;\ j_i \leq \ka_i,  \ i=1,2,3, \eps>0, \label{defBSP-vf}
\end{gather}
where $V_j,$ $j=1,2,3$ are $\CI$ vector fields such that $V_j$ is tangent to $\Sigma_k,$ $k\not=j$ but not tangent to $\Sigma_j.$   
\end{prop}
\begin{proof}   It is obvious that if \eqref{defBSP-vf} holds, then $u\in H^{s-, \ka_1,\ka_2,\ka_3}(U).$ To prove the converse we may replace $u$ with $\lan D_y\ran^{s-\eps} u$ and assume that 
$\lan \eta_1\ran^{k_1}\lan \eta_2\ran^{k_2}\lan \eta_3\ran^{k_3} \widehat{u} \in L^2(\mr^3).$
We will prove the following 
\begin{lemma}\label{iter-reg} If $(y_1\p_{y_1})^{m_1}(y_2\p_{y_2})^{m_2}(y_3\p_{y_3})^{m_3} u\in L^{2}(\mr^3)$ for all $m_j\in \mn,$ $j=1,2,3,$ and if $\lan D_{y_1}\ran^k u \in L^2(\mr^3)$ for some $k>0,$ then for all $\eps>0,$
\begin{gather}
(y_1\p_{y_1})^{m_1}(y_2\p_{y_2})^{m_2}(y_3\p_{y_3})^{m_3} \lan D_{y_1}\ran^{k-\eps} u\in L^{2}(\mr^3). \label{it-reg1}
\end{gather}
\end{lemma}
\begin{proof}
 By taking Fourier transform in $(y_1,y_2,y_3),$ this is equivalent to showing that for $v=\widehat{u},$
\begin{gather}
\begin{gathered}
\text{ if } \lan \eta_1\ran^k v \in L^2(\mr^3) \text{ and  }
(\eta_1\p_{\eta_1})^{m_1}(\eta_2\p_{\eta_2})^{m_2}(\eta_3\p_{\eta_3})^{m_3} v \in L^2(\mr^3), m_j\in \mn,  j=2,3,\\
\text{  then for all } \eps>0,
(\eta_1\p_{\eta_1})^{m_1}(\eta_2\p_{\eta_2})^{m_2} (\eta_3\p_{\eta_3})^{m_3} \lan \eta_1\ran^{k-\eps} v \in L^2(\mr^3), 
m_j \in \mn, \ j=1,2,3.
\end{gathered}\label{der-estimate}
\end{gather}
 Once we to prove \eqref{der-estimate} for $m_2=m_3=0;$ we can replace $v$ by $\lan \eta_1\ran^{k-\eps} (\eta_1\p_{\eta_1})^{m_1} v,$ and use the same argument to prove the general case.  So we just assume that $m_2=m_3=0.$ 
If $v\in C_0^\infty(\mr^3)$ then integration by parts give that
\begin{gather*}
||(\eta_1 \p_{\eta_1})^j \lan \eta_1\ran^{\frac{k}{2}} v||_{L^2}^2= 
(-1)^j \int_{\mr^3} \left((\p_{\eta_1} \eta_1)^j (\eta_1\p_{\eta_1})^j \lan \eta_1\ran^{\frac{k}{2}}v\right) (\lan \eta_1\ran ^{\frac{k}{2}} \overline{v}) \ d\eta_1d\eta_2 d\eta_3,
\end{gather*}
Since $|\eta_1\p_{\eta_1}\lan \eta_1\ran|= \eta_1^{2} \lan \eta_1\ran^{-1} \leq \lan \eta_1\ran,$ it is not difficult to use induction to prove that the commutator $[(\eta_1\p_{\eta_1})^{m}, \lan \eta_1\ran^{r}]$ satisfies $
|[(\eta_1\p_{\eta_1})^{m}, \lan \eta_1\ran^{r}] u|\leqs  \lan \eta_1\ran^r\sum_{l=0}^m |(\eta_1\p_{\eta_1})^{l} u|.$
Therefore we obtain
\begin{gather}
\begin{gathered}
||(\eta_1\p_{\eta_1})^j \lan \eta_1\ran^{\frac{k}{2}} v||_{L^2}^2 \leqs  \sum_{l=0}^{2j} \int_{\mr^3}
|(\eta_1\p_{\eta_1})^l  v| \lan \eta_1\ran^k |v| \ d\eta \leqs \left(\sum_{l=1}^{2j} ||(\p_{\eta_1}\eta_1)^l v||_{L^2}\right) || \lan \eta_1\ran^k v||_{L^2}. \label{ineq-v}
\end{gathered}
\end{gather}

We claim that for any $r\in \mn$ there exists a sequence  $\psi_{N,r}(\eta_1)\in C^\infty(\mr),$ $N\in \mn$  such that  
\begin{gather}
\begin{gathered}
\psi_{N,r}(\eta_1)=1 \text{ if } \eta_1<N-2  \text{ and } \psi_{N,r}(\eta_1)=0 \text{ if } \eta_1>2N+2 \\ \text{ and moreover } 
|\eta_1^m( \p_{\eta_1})^m\psi_{N,r}|\leq C(m,r) \text{ for } m \leq r.
\end{gathered}\label{est-gn}
\end{gather}
To prove this claim we start by constructing a sequence of polynomials $g_{N,r}$  of degree $2r+2$ such that
\begin{gather*}
g_{N,r}(N)=1, \; g_{N,r}(2N)=0, \text{ and provided that } 1\leq m \leq r, \\
\text{ their derivatives satisfy }  g_{N,r}^{(m)}(N)=g_{N,r}^{(m)}(2N)=0  
\text{ and } |g_{N,r}^{(m)}(s)| \leq \frac{C(m,r)}{N^m} \text{ for } s\in [N,  2N].
\end{gather*}
We pick $g_{N,r}$ such that its first derivative satisfies
\begin{gather*}
g_{N,r}'(s) =\frac{A}{N^{2r+2}} s(s-N)^r(s-2N)^{r}= \frac{A}{N^{2r+2}} (s-N)^{r+1}(s-2N)^r + \frac{A}{N^{2r+1}} (s-N)^r(s-2N)^r.
\end{gather*}
It is clear that for $1\leq m \leq r,$ $g_{N,r}^{(m)}(N)=g_{N,r}^{(m)}(2N)=0$ and  $|g_{N,r}^{(m)}(s)| \leq \frac{C(m,r)}{N^m}$  for 
$s\in [N,  2N].$ 
 
 We need to show one can pick $g_{N,r}$ to satisfy $g_{N,r}(N)=1, \; g_{N,r}(2N)=0.$  We integrate $g_{N,r}'(s)$ by parts and choose $g_{N,r}(s)$ to be
\begin{gather*}
g_{N,r}(s)=\frac{A}{N^{2r+2}} \sum_{j=0}^{r+1} C_{j,r}(s-N)^{r+1-j}(s-2N)^{r+j+1} + \frac{A}{N^{2r+1}} \sum_{j=0}^{r} D_{j,r}(s-N)^{r-j}(s-2N)^{r+j+1}, \\
\text{ where } C_{j,r}=(-1)^j \frac{(r+1)!}{(r+1-j)!} \frac{r!}{(r+j+1)!}, \;  D_{j,r}=(-1)^j \frac{r!}{(r-j)!} \frac{r!}{(r+j+1)!}.
\end{gather*}
 It is clear that $g_{N,r}(2N)=0$ and we pick $A$ such that $g_{N,r}(N)=1.$  This gives
 \begin{gather*}
 A= (-1)^{r+1} \frac{(2r+2)!}{3(r+1)(r!)^2}.
 \end{gather*}

  Let $f_{N,r}(s)$ be defined as 
\begin{gather*}
f_{N,r}(s)= 1 \text{ if } S\leq N, \\
f_{N,r}(s)= g_{N,r} \text{ if } N< s\leq 2N \\
f_{N,r}(s)=0 \text{ if } s>2N.
\end{gather*}

Let $\chi\in C_0^\infty(\mr)$ such that $\int \chi(s) ds=1,$ $\chi(s)=1$ if $|s|<1$ and $\chi(s)=0$ if $|s|>2$ and let $\psi_{N,r}(s)=\chi\star f_{N,r}(s).$  Then
\begin{gather*}
\psi_{N,r}(\eta_1)= \int_{\mr}  f_{N,r}(\eta_1-t)\chi(t) \ dt= \int_{\eta_1-2N}^{\eta_1-N} \chi(t)g_{N,r}(\eta_1-t)\ dt+\int_{\eta_1-N}^\infty \chi(t) \ dt.
\end{gather*}
Since $\chi(t)=0$ for $t>2,$ we have that $\psi_{N,r}(\eta_1)=0$ if $\eta_1-2N>2.$  Similarly, since $\chi(t)=0$ for $t<-2,$  we deduce that $\psi_{N,r}(\eta_1)=\int \chi(t) dt=1,$ provided  $\eta_1-N<-2.$

Using that $g_{N,r}(N)=1$ and $g_{N,r}(2N)=0$ we find that
\begin{gather*}
\p_{\eta_1}\psi_N(\eta_1)=  \int_{\eta_1-2N}^{\eta_1-N} \chi'(t)g_{N,r}(\eta_1-t)\ dt+\int_{\eta_1-N}^\infty \chi'(t) \ dt=\\
 \int_{\eta_1-2N}^{\eta_1-N} [\chi(t)g_{N,r}(\eta_1-t)]'\ dt+  \int_{\eta_1-2N}^{\eta_1-N} \chi(t)g_{N,r}'(\eta_1-t)\ dt- \chi(\eta_1-N)=\\
\chi(\eta_1-N) g_{N,r}(N)-\chi(\eta_1-2N) g_{N,r}(2N) +\int_{\eta_1-2N}^{\eta_1-N} \chi(t)g_{N,r}'(\eta_1-t)\ dt- \chi(\eta_1-N)=\\
\int_{\eta_1-2N}^{\eta_1-N} \chi(t)g_{N,r}'(\eta_1-t)\ dt.
\end{gather*}
 But if  $t\in [\eta_1-2N, \eta_1-N],$ then $\eta_1-t\in [N, 2N]$ and hence $|g_{N,r}'(\eta_1-t)|\leq C/N$ therefore $|\eta_1 \p_{\eta_1} \psi_{N,r}|\leq C.$  
 
 Now we have to verify the condition for the second derivative.  In general, using the  fact that $g_{N,r}(N)=1$ and that $g_{N,r}(2N)=0,$  we have
\begin{gather*}
 \p_{\eta_1}^j\psi_{N,r}(\eta_1)=  \int_{\eta_1-2N}^{\eta_1-N} (\p_t^{j-1} \chi(t)) g_{N,r}'(\eta_1-t)\ dt.
\end{gather*}
Integrating by parts and using that $g_{N,r}^{(m)}(N)=g_{N,r}^{(m)}(2N)=0$ for $m\leq r,$ we have that
\begin{gather*}
 \p_{\eta_1}^j\psi_{N,r}(\eta_1)=  \int_{\eta_1-2N}^{\eta_1-N}  \chi(t) g_{N,r}^{(j)} (\eta_1-t)\ dt, \;\ j\leq r.
\end{gather*}
But in this interval $|g_{N,r}^{(m)}(\eta_1-t)|\leq \frac{C(m,r)}{N^M},$ provided $1\leq m\leq r,$ and so 
$|\eta_1^j\p_{n_2}^j \psi_{N,r} |\leq C(j,r)$ for $j\leq r.$   This shows the existence of $\psi_{N,r}$ satisfying \eqref{est-gn}.

Now take a function $v$ which satisfies \eqref{der-estimate} and apply \eqref{ineq-v} to 
$v_N=\psi_{N,r} (\eta_1) \vphi_N(\eta_2,\eta_3)v$ where $\vphi_N\in C_0^\infty(\mr^2),$ with $\vphi_N(\eta_1,\eta_3)=1$ if $|(\eta_1,\eta_3)|<N$ and $\vphi_N(\eta_2,\eta_3)=0$ if $|(\eta_2,\eta_3)|>2N.$ 

Since for $r\geq l,$
\begin{gather*}
||(\eta_1\p_{\eta_1})^l v_N||_{L^2(\mr^3)} \leqs \sum_{m=1}^l || \vphi_N(\eta_2,\eta_3)( (\eta_1\p_{\eta_1})^m \psi_{N,r}(\eta_1)) (\eta_n \p_{\eta_1})^{l-m} v||_{L^2(\mr^3)}+ \\
||\vphi_N(\eta_2,\eta_3) \psi_{N,r}(\eta_1) (\eta_1\p_{\eta_1})^l v||_{L^2(\mr^3)},
\end{gather*}
 $(\eta_1\p_{\eta_1})^m\psi_N$ is supported in $[N-2,2N-2]$ and is uniformly bounded by a constant that does not depend on $N,$   and $v\in L^2$ it follows that the terms
 $| \vphi_N(\eta_2,\eta_3)( (\eta_1\p_{\eta_1})^m \psi_{N,r}) (\eta_n \p_{\eta_1})^{l-m} v||_{L^2(\mr^3)}$
 converge to $0$ as $N\rightarrow \infty.$  The term $||\vphi_N(\eta_1,\eta_3) \psi_N(\eta_1) (\eta_1\p_{\eta_1})^l v||$ converges to $|| (\eta_1\p_{\eta_1})^l v||$ as $N\rightarrow \infty.$ 
So we conclude that  $(\eta_1\p_{\eta_1})^j \lan \eta_1\ran^{\frac{k}{2}} v_N$ is a Cauchy sequence and converges to $(\eta_1\p_{\eta_1})^j\lan \eta_1\ran^{\frac{k}{2}} v$ in $L^2.$   So we have shown the following:
\begin{gather}
\begin{gathered}
\text{ if } \lan \eta_1\ran^k v ,  \ (\eta_1\p_{\eta_1})^j v \in L^2,  \  j\in \mn, \text{ then } 
\lan \eta_1 \ran^{\frac{k}{2}} (\eta_1 \p_{\eta_1})^j v \in L^2, \; j\in \mn.
\end{gathered}\label{Ind-1}
\end{gather}
We apply this result to $w=\lan \eta_1\ran^{\frac{k}{2}} v.$  We know that
\begin{gather*}
\lan \eta_1\ran^{\frac{k}{2}} w \in L^2 \text{ and } (\eta_1\p_{\eta_1})^j w \in L^2,  \  j\in \mn
\end{gather*}
and so we conclude that
\begin{gather*}
\lan \eta_1 \ran^{\frac{k}{4}} (\eta_1 \p_{\eta_1})^j w \in L^2, \; j\in \mn, \text{ and hence} \\
\lan \eta_1 \ran^{\frac{k}{2}+ \frac{k}{4}} (\eta_1 \p_{\eta_1})^j  v \in L^2, \; j\in \mn.
\end{gather*}
Again, we apply the same argument to $w=\lan \eta_1 \ran^{\frac{k}{2}+ \frac{k}{4}} v$ and we conclude that
\begin{gather*}
\lan \eta_1 \ran^{\frac{k}{2}+ \frac{k}{4}+\frac{k}{8}} (\eta_1 \p_{\eta_1})^j  v \in L^2, \; j\in \mn.
\end{gather*}
Therefore,  after $J$ steps we obtain
\begin{gather*}
\lan \eta_1 \ran^{k(1-2^{J-1})} (\eta_1 \p_{\eta_1})^j  v \in L^2, \text{ for all }  j,J \in \mn.
\end{gather*}
This proves the Lemma.
\end{proof}
To prove the proposition,  we notice that if 
$\lan \p_{y_1}\ran^{k_1}\lan \p_{y_2}\ran^{k_2}\lan \p_{y_3} \ran^{k_3} u \in H_{\loc}^{s-\eps}(\Omega),$ then we also have $\lan \p_{y_1}\ran^{k_1+\del}\lan \p_{y_2}\ran^{k_2+\del }\lan \p_{y_3} \ran^{k_3+\del} u \in H_{\loc}^{s-\eps-3\del}(\Omega),$ for some $\del>0.$ If $u \in I H^{s-}_{\loc}(\Omega,\mcv_{123}),$ then it follows from the Lemma that if $k_j \in \mn,$ $j=1,2,3,$ then
\begin{gather*}
(\p_{y_1}+ a_{12} y_{2}\p_{y_2}+ a_{13} y_3\p_{y_3})^{k_1} (\p_{y_2}+ a_{21} y_{1}\p_{y_1}+ a_{23} y_3\p_{y_3})^{k_2}(\p_{y_3}+ a_{31} y_{1}\p_{y_1}+ a_{32} y_2\p_{y_2})^{k_3} \in H^{s-}_{\loc}(\Omega).
\end{gather*}
This ends the proof of the proposition.
\end{proof}

This shows that  in view of Theorem \ref{REGSD}, if $u$ satisfies \eqref{Weq}, then in 
 $\Omega^-=\Omega\cap \{t<0\},$  we can work on any convenient local coordinates $(y_1,y_2,y_3)$ for which 
 $\Sigma_j=\{y_j=0\}.$

  The following will be very useful below:

\begin{prop}\label{BSP2} (M. Beals \cite{Beals4})   Fix local coordinates $y=(y_1,y_2,y_3)$ valid in a neighborhood $\Omega$ of $0,$ if $s\geq 0$ and $\min\{ s+ k_j\}>\ha,$ then $H_{\loc}^{s-, k_1,k_2,k_3 }(\Omega)$  is closed under multiplication.   Moreover, if $a_j>0$ and $a_1+a_2+a_3=1,$ then
\begin{gather}
\begin{gathered}
H_{\loc}^{s+1-, k_1,k_2,k_3}(\Omega) \subset H_{\loc}^{s-, k_1+a_1,k_2+a_2,k_3+a_3}(\Omega). 
\end{gathered}\label{inclusion}
\end{gather}
\end{prop}

We analyze mapping properties of the fundamental solution of $\square$ acting on Beals spaces.   First we analyze the action of $E_+$ on functions supported in a neighborhood of the point of triple interaction. 
It will be convenient to work with the operator in a special form:
\begin{prop}\label{norm-form} There exist coordinates $y=(y_1,y_2,y_3)$ near $\{0\}$ such that $\Sigma_j=\{y_j=0\},$ $j=1,2,3$ and
\begin{gather}
\begin{gathered}
\square=  a_{12}(y) \p_{y_1}\p_{y_2}+ a_{13}(y) \p_{y_1}\p_{y_3} + a_{23}(y) \p_{y_2}\p_{y_3}+ 
  \mcl, \\
\mcl \text{ is a differential operator of order one and }  a_{ij}(y) \in C^\infty, \;\ i,j=1,2,3.
\end{gathered}\label{mod-oper}
\end{gather}
 \end{prop}
\begin{proof}  Since $\Sigma_j$ is characteristic for $\square,$ $j=1,2,3,$  one must have
\begin{gather}
\begin{gathered}
\square= b_{11} (y) y_1 \p_{y_1}^2 + b_{12}(y) \p_{y_1}\p_{y_2}+ b_{13}(y) \p_{y_1}\p_{y_3}  + b_{22}(y)y_2 \p_{y_2}^2 + b_{23}(y) \p_{y_2}\p_{y_3}+  b_{33}(y) y_3 \p_{y_3}^2+
 \\ \sum_{j=1}^3 b_j(y) \p_{y_j}+b(y), \;\ 
b_{ij}(y), b_j(y), b(y) \in C^\infty, \;\ i,j=1,2,3.
\end{gathered}\label{mod-oper-1}
\end{gather}
The strict hyperbolicity requires $b_{12}(0) b_{13}(0) b_{23}(0)\not=0,$ and hence $b_{12}(y) b_{13}(y) b_{23}(y)\not=0$ near $0.$  Any change of variables that preserve $\Sigma_j=\{y_j=0\}$ must be of the form 
\begin{gather*}
Y_j = y_j f_j(y), \ j=1,2,3, \;  f_j(y)\not=0 \text{ near } 0,
\end{gather*}
and therefore,
\begin{gather*}
\p_{y_1}= (f_1+ y_1\p_{y_1} f_1) \p_{Y_1}+ y_2\p_{y_1} f_2 \p_{Y_2}+ y_3\p_{y_1}f_3 \p_{Y_3}, \\
\p_{y_2}= y_1\p_{y_2}f_1 \p_{Y_1}+ (f_2+ y_2\p_{y_2} f_2) \p_{Y_2}+ y_3\p_{y_2}f_3 \p_{Y_3},  \\
\p_{y_3}= y_1\p_{y_3}f_1 \p_{Y_1}+ y_2\p_{y_3}f_2 \p_{Y_2}+ (f_3+ y_3\p_{y_3} f_3) \p_{Y_3}.
\end{gather*}
This means that \eqref{mod-oper-1} transforms into
\begin{gather*}
\square = y_1A \p_{Y_1}^2 + y_2 B \p_{Y_2}^2+ y_3 C \p_{Y_3}^2 + A_{12} \p_{Y_1}\p_{Y_2} + A_{13} \p_{Y_1}\p_{Y_3}+ A_{23} \p_{Y_2}\p_{Y_3} + \wt{\mcl}(Y,\p_y), \\
\text{ where } \wt{\mcl} \text{ is of order one } \text{ and for } M_j= f_j+ y_j \p_{y_j} f_j, \\
A=A(y,f_1, \nabla_y f_1)=  \\ b_{11} M_1^2+ b_{12} M_1 \p_{y_2} f_1+ b_{13}M_1 \p_{y_3} f_1+ y_1y_2 b_{22} (\p_{y_2} f_1)^2+ y_1 b_{23}\p_{y_2}f_1\p_{y_3} f_1+ y_1y_3 b_{33} (\p_{y_3} f_1)^2,\\
B=B(y,f_2, \nabla_y f_2)=  \\ y_1 y_2 b_{11} (\p_{y_1} f_2)^2+ b_{12} M_2 \p_{y_1} f_2+ y_2 b_{13}\p_{y_1} f_2 \p_{y_3} f_2+  b_{22} M_2^2+ b_{23} M_2 \p_{y_3}f_2+ y_2y_3 b_{33} (\p_{y_3} f_2)^2,   \\ 
C=C(y,f_3, \nabla_y f_3)= \\ y_1 y_3 b_{11} (\p_{y_1} f_3)^2 + b_{13} M_3 \p_{y_1} f_3+ y_3 b_{12} \p_{y_1} f_3 \p_{y_2} f_3+  y_2 y_3 b_{22} (\p_{y_2} f_3)^2 + b_{23} M_3 \p_{y_2} {f_3}+ b_{33} M_3^2
\end{gather*}
 We want to find $f_1,f_2$ and $f_3$ such that 
 \begin{gather*}
 A(y,f_1, \nabla_y f_1)=B(y,f_2, \nabla_y f_2)=C(y,f_3, \nabla_y f_3)=0. 
 \end{gather*}
  We claim that if $\del>0$ is small enough, there exist unique functions $f_j(y),$ $j=1,2,3,$  defined in  $\{y: |y|<\delta\}$,  such that
   \begin{gather}
  \begin{gathered}
A(y,f_1, \nabla_y f_1)=0 , \;\ B(y,f_2, \nabla_y f_2)=0, \;\ C(y,f_3, \nabla_y f_3)=0,  \text{ in } U\\
 f_1(y_1,0,y_3) = u_1(y_1,y_3), \;\ f_2(y_1,y_2,0)= u_2(y_1,y_2),  \;\  f_3(0,y_2,y_3)= u_3(y_2,y_3),
 \end{gathered}\label{NLS}
  \end{gather}
privided the initial data  $u_j \in \CI$ and $|u_j(0,0)|\not=0.$  Let us consider the case of $A(y,f_1,\nabla_y f_1),$ the others are of course the same.  This is a non-linear first order pde, which can be solved by the method of chracateristics, see for example the book by Evans \cite{Evans}. We need to verify that the problem is non-characteristic and the necessary compatibility conditions are verified.  As usual, let $u=f_1,$ $\p_{y_j} f_1=p_j,$ $j=1,2,3.$ In these variables, the compatibility conditions at $\{y_2=0\}$ are
\begin{gather*}
p_1= \p_{y_1} u_1,\\
p_3= \p_{y_3} u_1,\\
A(y,u,p)= b_{11} M_1^2+ M_1 (b_{12}p_2+ b_{13} p_3) + b_{23} y_1 p_2 p_3 + b_{13} y_1 y_3 p_3^2=0,
\end{gather*}
where at $\{y_2=0\},$ $M_1= u_1+y_1 \p_{y_1} u_1\not=0,$ for $\del>0$ small.  The last equation is linear in $p_2$ and can be uniquely solved, provided  if $|M_1 b_{12} +  y_1 b_{23} p_3|>\eps$ for $|(y_1,y_2)|<\del.$ This can be arranged, since $b_{12}(0)\not=0$ and $u_1(0,0)\not=0.$ The equation is non-characteristic with respect to $\{y_2=0\}$  if $\p_{p_2} A(y,u,p)\not=0,$ provided $|(y_1,y_2)|<\del.$ But at $\{y_2=0\},$
\begin{gather*}
\p_{p_2} A(y,u,p)|_{\{y_2=0\}}=M_1 b_{12}+ y_1 b_{23} p_3.
\end{gather*}
Again, using that $b_{12}(0)\not=0$ and $u_1(0,0)\not=0,$ one can choose $\del>0$ small enough so that  
$\p_{p_2} A(y,u,p)\not=0$ at $\{y_2=0\}.$ Therefore, there exists a unique $f_1(y)$ in a neighborhood of $\{0\}$ that satisfies
$A(y,f_1,\nabla_y f_1)=0$ and $f_1(y_1,0,y_3)= u_1(y_1,y_3).$ This ends the proof of the Proposition.
\end{proof}

\begin{prop}\label{BSP3}    Let  $\Omega$ be a neighborhood of $\{0\}$ and let $y=(y_1,y_2,y_3)$ be local coordinates in $\Omega$ such that \eqref{mod-oper} holds.  Let  $H_{\loc}^{s, k_1,k_2,k_3}(\Omega)$  be the space defined above with respect to this choice of coordinates and $k_j\geq 0.$  Let $E_+$ denote the forward fundamental solution to $\square.$  If  $\Omega$ is small enough and $\vphi, \psi \in C_0^\infty(\Omega)$ 
\begin{gather}
\psi E_+ \vphi : H^{s, k_1,k_2,k_3}(\mr^3)\longrightarrow H^{s+1, k_1,k_2,k_3}(\mr^3),  \label{mappinge+}
\end{gather}
and by that we mean 
\begin{gather}
\lan \eta_1\ran^{k_1} \lan \eta_2\ran^{k_2} \lan \eta_3 \ran^{k_3} \lan \eta \ran^{s} \mcf( \vphi G)\in L^2(\mr^3)   \Longrightarrow 
\lan \eta_1\ran^{k_1} \lan \eta_2\ran^{k_2} \lan \eta_3 \ran^{k_3} \lan \eta \ran^{s+1} \mcf( \psi E_+\vphi G)\in L^2(\mr^3).
\label{EST-EP} 
\end{gather}

\end{prop}
\begin{proof}  In the model case considered by M. Beals \cite{Beals4}, $\square= \p_{y_1}\p_{y_2}+ \p_{y_1}\p_{y_3}+\p_{y_2}\p_{y_3},$ this result is immediate  because $\square$ commutes with $\lan \p_{y_j}\ran $ and $\lan \p_y\ran$ and
$E_+: H_{\loc}^m(\mr^3) \longmapsto H_{\loc}^{m+1}(\mr^3).$   Here, we need to carefully analyze the commutators since the order of differentiation is important.

We first prove \eqref{mappinge+} for $k_j\in \mn_0=\mn\cup\{0\}$ and in this case we need to show that if $\square u=G\in H_{\loc}^{s, k_1,k_2,k_3}(\Omega),$ then 
\begin{gather}
 \p_{y_1}^{m_1} \p_{y_2}^{m_2} \p_{y_3}^{m_3} \vphi u \in H_{\loc}^{s+1}(\mr^3), \;\ 0\leq m_j \leq k_j, \;\ j=1,2,3.\label{mappinge+1}
\end{gather}

As usual, we denote $\p_y^\alpha=\p_{y_1}^{\alpha_1}\p_{y_2}^{\alpha_2} \p_{y_3}^{\alpha_3},$ $\alpha=(\alpha_1,\alpha_2,\alpha_3)\in \mn_0^3,$ and we analyze the commutator of $\p_y^\alpha$ and $\square.$  We claim that the following holds:
\begin{gather}
\begin{gathered}
\left[\p_y^\alpha, \square\right]= \p_y^\alpha \square-\square \p_y^\alpha= \sum_{\beta_j=0}^{\alpha_j} F_{\alpha,\beta} \p_y^\beta \square + \sum_{\beta_j=0}^{\alpha_j} \mcl_{\alpha,\beta} \p_{y}^\beta u 
\end{gathered}\label{comm-alpha}
\end{gather}
where $F_{\alpha,\beta}, q_{j,\alpha,\beta}\in C^\infty$ and $\mcl_{\alpha,\beta}$ is a first order differential operator.

We use induction to prove this formula, and we begin with the case $|\alpha|=1.$ To simplify the notation, we analyze  $\left[\p_{y_1},\square\right].$ It follows from \eqref{mod-oper} that
\begin{gather*}
\left[\p_{y_1},\square \right]=\p_{y_1} \square- \square \p_{y_1}=  (  (\p_{y_1} a_{12}) \p_{y_1} \p_{y_2}+ (\p_{y_1} a_{13}) \p_{y_1} \p_{y_3}+  (\p_{y_1} a_{23}) \p_{y_2} \p_{y_3} + [\p_{y_1},\mcl]).
\end{gather*}
We then use that
\begin{gather*}
\p_{y_2}\p_{y_3}=\frac{1}{a_{23}}\left(\square -a_{12} \p_{y_1}\p_{y_2}-a_{13} \p_{y_1}\p_{y_3}- \mcl\right)
\end{gather*}
and we obtain
\begin{gather*}
\left[ \p_{y_1} ,\square \right]=\frac{1}{a_{23}}(\p_{y_1}a_{23}) \square+ \mcl_{1,1}\p_{y_1}  + \mcl_{1,0}, \\
\mcl_{1,1}=
( (\p_{y_1} a_{12}) -\frac{a_{12}}{a_{23}} (\p_{y_1}a_{23}) )\p_{y_2} + ( (\p_{y_1} a_{13})-\frac{a_{13}}{a_{23}} (\p_{y_1}a_{23})) \p_{y_3}   
\end{gather*}
This proves \eqref{comm-alpha} for $\alpha=(1,0,0),$ and of course by symmetry similar formulas hold for $\alpha=(0,1,0)$ and $\alpha=(0,0,1).$

 Now assume  that \eqref{comm-alpha} holds for $\alpha=(\alpha_1,\alpha_2,\alpha_3)$ and we want to show that it holds for $\alpha=(\alpha_1+1,\alpha_2,\alpha_3).$   One can easily verify that
\begin{gather}
\left[ \p_{y_1} \p_y^\alpha, \square\right]= \p_{y_1} \left[ \p_{y}^{\alpha},\square\right] + \left[\p_{y_1}, \square\right] \p_{y}^{\alpha}. \label{der-comm}
\end{gather}

It follows from \eqref{comm-alpha} that 
\begin{gather}
  \begin{gathered}
 \p_{y_1}  \left[ \p_{y_1}^\alpha, \square\right] = \sum_{\beta_1=0}^{\alpha_1+1}\sum_{\beta_2=0}^{\alpha_2}\sum_{\beta_3=0}^{\alpha_3}  F_{\alpha,\beta} \p_y^\beta \square + 
 \sum_{\beta_j=0}^{\alpha_j-1} [\p_{y_1}, F_{\alpha,\beta}] \p_y^\beta \square+ \\  \sum_{\beta_1=0}^{\alpha_1+1}\sum_{\beta_2=0}^{\alpha_2}\sum_{\beta_3=0}^{\alpha_3}  \mcl_{\alpha,\beta} \p_y^\beta +  \sum_{\beta_j=0}^{\alpha_j}  [\p_{y_1}, \mcl_{\alpha,\beta}] \p_y^\beta 
\end{gathered} \label{term1}
 \end{gather}
We notice that  $[\p_{y_1},F_{\alpha,\beta}]\in C^\infty$ and  $[\p_{y_1},\mcl_{\alpha,\beta}]$ is a differential operator of order one. 
So we conclude that
\begin{gather}
\begin{gathered}
 \p_{y_1}  \left[ \p_{y_1}^\alpha, \square\right] = \sum_{\beta_1=0}^{\alpha_1+1}\sum_{\beta_2=0}^{\alpha_2}\sum_{\beta_3=0}^{\alpha_3}  \widetilde F_{\alpha,\beta} \p_y^\beta \square + 
   \sum_{\beta_1=0}^{\alpha_1+1}\sum_{\beta_2=0}^{\alpha_2}\sum_{\beta_3=0}^{\alpha_3} \widetilde  \mcl_{\alpha,\beta} \p_y^\beta 
 \end{gathered}\label{term2NN}
  \end{gather}
 with $\widetilde F_{\alpha,\beta}\in C^\infty$ and $\widetilde \mcl_{\alpha,\beta}$ a differential operator of order one.

 On the other hand, using  \eqref{comm-alpha} for $|\alpha|=1,$ we deduce that the second term in \eqref{der-comm} is equal to
\begin{gather*}
\left[\p_{y_1}, \square\right] \p_{y}^{\alpha}= F_{1,0} \square \p_y^\alpha+\sum_{j=0}^1 \mcl_{1,j} \p_{y_1}^j \p_y^\alpha
\end{gather*}
But by assumption we have 
\begin{gather}
\begin{gathered}
F_{1,0} \square \p_y^\alpha= F_{1,0}\p_y^\alpha \square + F_{1,0}\left[ \square, \p_y^\alpha\right]= \\
F_{1,0}\p_y^\alpha \square - \sum_{\beta_j=0}^{\alpha_j} F_{1,0} F_{\alpha,\beta} \p_y^\beta \square - \sum_{\beta_j=0}^{\alpha_j} F_{1,0}\mcl_{\alpha,\beta} \p_{y}^\beta u  
\end{gathered}\label{term2}
\end{gather}
Now we substitute \eqref{term1} and \eqref{term2} in \eqref{der-comm}, and we  conclude  that \eqref{comm-alpha} holds for all $\alpha\in \mn_0^3.$ 

Now we use the commutator formula \eqref{comm-alpha} to prove \eqref{mappinge+1}. We start with the case $|\alpha|=1.$   If  $\square u=G,$  then it follows from \eqref{comm-alpha} that
\begin{gather*}
G=\square u, \\
\p_{y_1}G-F_{1,0} G= (\square +\mcl_{1,1}) \p_{y_1} u +  \mcl_{1,0} u.
\end{gather*}
So if we  let $\mcu_{1}=(u, \p_{y_1} u)$ and $\mcg_1=(G, \p_{y_1} G-F_{1,0}G),$  we get a $2\times 2$ system of equations
  \begin{gather*}
  \mcp_1 \mcu_1=\mcg_1, 
  \end{gather*}
  where $\mcp_1=(p_{ij})_{1\leq i,j \leq 2}$ is a $2\times 2$ matrix of differential operators with $p_{11}=\square,$ $p_{12}=0,$ $p_{21}= \mcl_{1,0}$ and 
  $p_{22}= \square +\mcl_{1,1}.$  The principal part of the operator  $\mcp_1$ is a diagonal matrix  $\square \Id_{2\times 2},$ and hence it is strictly hyperbolic.  If $G\in H_{\loc}^{s,1,0,0},$ then $\mcg_1\in H_{\loc}^{s},$ and so $\mcu_1\in H_{\loc}^{s+1}$ which implies that $u \in H_{\loc}^{s+1,1,0,0}.$ 
  
   In general, if $G\in H_{\loc}^{s,k_1,k_2,k_3},$ let  $\mcu_{k_1,k_2,k_3}=(u, \p_{y}^\alpha u),$ $\alpha_j\leq k_j,$ and $\mcg_{k_1,k_2,k_3}= (G, \p_y^\alpha G),$ $\alpha_j\leq k_j,$  $j=1,2,3,$ we get a system

\begin{gather*}
\mcp_{k_1,k_2,k_3} \mcu_{k_1,k_2,k_3}=\mcm_{k_1,k_2,k_3} \mcg_{k_1,k_2,k_3},
\end{gather*}
where $\mcm_{k_1,k_2,k_3}$ is a matrix of $\CI$ functions and $\mcp_{k_1,k_2,k_3}$ is a matrix of linear differential operators with diagonal principal part $\square \Id_{m\times m},$ where $m$ is the number of entries of $\mcu_{k_1,k_2,k_3}.$   This is a strictly hyperbolic system and this proves the proposition for $k_j\in \mn.$

Next, to prove \eqref{mappinge+} for $k_j\in \mr_+,$ $j=1,2,3.$  The characterization of  $H^{s, k_1,k_2,k_3}(\mr^3)$ in terms of the Fourier transform $\mcf$ is
\begin{gather*}
u\in H^{s, k_1,k_2,k_3}(\mr^3) \text{ if and only if } \mcf u \in L^2(\mr^3, \mu_{s,k_1,k_2,k_3}(\eta) d\eta), \\
\text{ where } \mu_{s,k_1,k_2,k_3}(\eta)=\lan \eta_1\ran^{2k_1} \lan \eta_2\ran^{2k_2} \lan \eta_3\ran^{2k_3}\lan \eta\ran^{2s}.
\end{gather*}
We want to show that if $T=\psi E_+\vphi,$ then 
\begin{gather}
\mcf \p_{y_j}T  \mcf^{-1}: L^2(\mr^3, \mu_{s,k_1,k_2,k_3}(\eta) d\eta) \longrightarrow L^2(\mr^3, \mu_{s,k_1,k_2,k_3}(\eta) d\eta), \ j=1,2,3 \label{interp-0}
\end{gather}
is a bounded linear operator for $k_j \in \mr_+.$ We have proved this statement for $k_j \in \mn_0.$  Fix $k_1, k_2 \in \mn_0,$ and for $m_1, m_2\in \mn_0,$ $r \in [0,1],$ set 
$k_3=r m_1+(1-r) m_2.$ We know the operator is bounded for $r=0$ and $r=1,$ so it follows from the Stein-Weiss Interpolation Theorem, see \cite{SteWei},  that \eqref{interp-0} holds for 
$r\in (0,1)$ and therefore the result holds for $k_1,k_2\in \mn_0$ and $k_3\in \mr_+.$
Now fix $k_3 \in \mr_+,$ $k_2\in \mn_0,$ and repeat the argument for $k_1=r m_1+(1-r) m_2,$ $m_1,m_2 \in \mn_0,$  $r\in [0,1].$  We conclude that \eqref{interp-0} holds for $k_2\in \mn_0$ and $k_1,k_3 \in \mr_+.$  We apply the same argument for $k_1,k_3 \in \mr_+$ fixed and $k_2=r m_1+(1-r) m_2,$ $r\in [0,1],$ $m_1,m_2\in \mn_0$ and we obtain the desired result.
\end{proof}

We will also need the analogue of Proposition \ref{BSP3} near the double intersections. In this case we work in a neighborhood of a point $q\in \Sigma_1\cap\Sigma_2,$ then the proof of Proposition \ref{norm-form} applies to show that there exist local coordinates $y=(y_1,y_2,y_3)$ valid near $q$ such that for  $\{|(y_1,y_2)|<\del\},$ and $\del$ small, $\Sigma_j=\{y_j=0\},$ $j=1,2,$  and  the operator $\square$ is given by
\begin{gather}
\square=  a_{12} \p_{y_1} \p_{y_2}  +\mcl(y,\p_y) \p_{y_3} + \mcl_1(y,\p_y),\label{norm-form-TS}
\end{gather}
where $\mcl(y,\p_y)$ and $\mcl_1(y,\p_y)$ are differential operators of order one. 
\begin{prop}\label{BSP4}    Let  $U\subset \Omega$ be an open subset such that $U\cap \Sigma_1 \cap \Sigma_2\not=\emptyset$  and suppose that  $y=(y_1,y_2,y_3)$ are local coordinates in $U$ such  that $\Sigma_1=\{y_1=0\},$ $\Sigma_2=\{y_2=0\}$ and \eqref{norm-form-TS} is valid in U.  Let  $H_{\loc}^{s, k_1,k_2,k_3}(U)$  be the space defined above with respect to this choice of coordinates.   If  $\vphi, \psi \in C_0^\infty(U),$ 
 if $E_+$ denotes the forward fundamental solution to $\square,$  and if $k_j\geq 0,$ $j=1,2,3,$ then
\begin{gather}
\psi E_+ \vphi : H^{s, k_1,k_2,k_3+\max\{k_1,k_2\}}(\mr^3)\longrightarrow H^{s+1, k_1,k_2, k_3}(\mr^3).  \label{mappinge1+}
\end{gather}
In particular, since $U$ does not depend on $k_j,$ $j=1,2,3,$ this implies that $\psi E_+ \vphi : H^{s, k_1,k_2,\infty}(\mr^3)\longrightarrow H^{s+1, k_1,k_2, \infty}(\mr^3).$
\end{prop}
\begin{proof}   In this case, the necessary commutator formula is 
\begin{gather}
\begin{gathered}
\left[ \p_{y}^\alpha, \square \right]= \sum_{\beta_j=0}^{\alpha_j} F_{\alpha,\beta} \p_y^\beta \square + 
\sum_{\beta_j=0}^{\alpha_j} \mcl_{\alpha,\beta} \p_{y}^\beta u + \\
\sum_{\beta_1=0}^{\alpha_1}\sum_{\beta_2=0}^{\alpha_2-1}\sum_{\beta_3=0}^{\alpha_3-1} \mcl_{1,\alpha,\beta} \p_{y_3} \p_y^\beta+
\sum_{\beta_1=0}^{\alpha_1-1}\sum_{\beta_2=0}^{\alpha_2}\sum_{\beta_3=0}^{\alpha_3-1} \mcl_{2,\alpha,\beta} \p_{y_3}\p_y^\beta+
\sum_{\beta_1=0}^{\alpha_1-1}\sum_{\beta_2=0}^{\alpha_2-1}\sum_{\beta_3=0}^{\alpha_3} \mcl_{3,\alpha,\beta} \p_{y_3}\p_y^\beta.
\end{gathered}\label{comm-alpha12}
\end{gather}
where $F_{\alpha,\beta}, q_{j,\alpha,\beta}, A_{\alpha,\beta} \in C^\infty$ and $\mcl_{\alpha,\beta}, {\mcl}_{j,\alpha,\beta},$ $j=1,2,3$  are  first order differential operators with $\CI$ coefficients.  In view of the third term on the second line of  from \eqref{comm-alpha12},  if one wants to form a system including
$\p_{y_1}^{k_1}\p_{y_2}^{k_2} \p_{y_3}^{k_3}u$ one needs to include  not only terms like
$\p_{y_1}^{\alpha_1}\p_{y_2}^{\alpha_2} \p_{y_3}^{\alpha_3}u,$ with $\alpha_j \leq k_j,$ but also
  and $\p_{y_1}^{k_1-1}\p_{y_2}^{k_2-1} \p_{y_3}^{k_3+1}u.$ Which in turn requires the inclusion of 
$\p_{y_1}^{k_1-2}\p_{y_2}^{k_2-2} \p_{y_3}^{k_3+2}u$ and so on up to $\p_{y_3}^{k_3+\max\{k_1,k_2\}}u.$ This proves the result for $k_j\in \mn.$  The general case follows by interpolation, as in the proof of Proposition \ref{BSP3}.
\end{proof}

\subsection{Products of Conormal Distributions}\label{prod-con-00}   Here we again assume $\Omega\subset \mr^3,$ since  Beals spaces are only defined in $\mr^3.$ Let $\Sigma_j=\{y_j=0\}$ and let $u_j \in I^{m_j-\frac{n}{4}+\ha}(\Omega,\Sigma_j),$ $j=1,2,3.$  We want to analyze the products $u_j u_k$ and  $u_1u_2u_3.$ This has been considered by several people including Eswarathasan \cite{Suresh}, Greenleaf and Uhlmann \cite{GreUhl},   Joshi \cite{Joshi}, Melrose and Uhlmann \cite{MelUhl} where careful symbolic expansions were established.  Here we use Beals spaces to control the lower order terms.  First we establish a relationship between Beals spaces and conormal distributions.
\begin{prop}\label{inc-con} If $\Sigma_1=\{y_1=0\}$ and $u\in I^{m-\oq}(\Omega,\Sigma_1)$ and $m<0,$ then
\begin{gather}
u \in H_{\loc}^{s,k_1,\infty,\infty}(\Omega), \text{ provided }  \text{ and } s+k_1 <-m-\ha. \label{inc-con1}
\end{gather}
\end{prop}
\begin{proof}  Let $\vphi\in C_0^\infty(\Omega),$ then according to \eqref{left-red}
\begin{gather*}
\vphi u-\mce= \int_{\mr} e^{i y_1\eta_1} a(\eta_1,y_2,y_3) d\eta_1, \;\ \mce\in C_0^\infty, a\in S^m(\mr\times \mr^{2}), \text{ compactly supported in } (y_2,y_3).
\end{gather*}
Therefore,
\begin{gather*}
\mcf(\vphi u-\mce)=  \mcf_{y''} a(\eta_1,\eta''), \text{ is rapidly decaying in } \eta''=(\eta_2,\eta_3),
\end{gather*}
where $\mcf_{y''}$ denotes the Fourier transform in $y''.$ We want to analyze the integral
\begin{gather*}
I=\int_{\mr^3} (1+|\eta|^2)^s (1+|\eta_1|)^{2k_1}(1+|\eta_2|)^{2k_2}(1+|\eta_3|)^{2k_3}  |\mcf(\vphi u-\mce)|^2 d\eta.
\end{gather*}
If $s<0,$ then using that $\mcf_{y''} a(\eta_1,\eta'')$ is rapidly decaying in $\eta'',$ for any $k_2$ and $k_3,$
\begin{gather*}
I \leq C \int_{\mr} (1+|\eta_1|)^{2k_1+2m} d\eta_1,
\end{gather*}
and this converges if $2 k_1+2m <-1$ or $k_1+m<-\ha.$ 

 If $s\geq 0,$ then for any $k_2$ and $k_3,$
 \begin{gather*}
 I \leq \int_{\mr^3} (1+|\eta''|^2)^s (1+|\eta_1|)^{2k_1+2s} |\mcf(\vphi u-\mce)|^2 d\eta\leq C \int_{\mr} (1+|\eta_1|)^{(2k_1+2s+2m)} d\eta_1,
 \end{gather*}
 and this converges if $2k_1+2m+2s<-1.$
\end{proof}

\begin{prop}\label{prod-three-dist}  Let  $\Omega\subset \mr^n$ be an open neighborhood of the origin and $y=(y_1,y_2,y_3)$  be coordinates in $\Omega$  such that $\Sigma_j=\{y_j=0\},$ $j=1,2,3,$  and let  
$v_j \in I^{m-\frac{n}{4}+\ha}(\Omega, \Sigma_j),$ $j=1,2,3,$ $m<-1.$  If 
\begin{gather*}
a_1(\eta_1,y_2,y_3), a_2(\eta_2,y_1, y_3), a_3(\eta_3,y_1,y_2)  \in  S^m(\mr \times \mr^{2}), \\
\text{  are respectively the principal symbols of } v_1, v_2 \text{ and } v_3
\end{gather*}
 then
\begin{gather}
\begin{gathered}
v_1 v_2 = w_{12} +  \mce_{12}, \ v_1 v_3 = w_{13} +  \mce_{13}, \text{ and }
v_2 v_3 = w_{23} +  \mce_{23},
\end{gathered}\label{prod-2}
\end{gather}
 where
 \begin{gather}
 \begin{gathered}
 w_{12}(y)= \int_{\mr} e^{i( \eta_1 y_1+ y_2\eta_2)} a_1(\eta_1,0, y_3) a_2( \eta_2,0, y_3)  \ d\eta_1d\eta_2, \\ \mce_{12} \in  H_{\loc}^{0-, -m+\ha,-m-\ha,\infty}(\Omega)+ H_{\loc}^{0-, -m-\ha,  -m+\ha ,\infty}(\Omega), \\
 w_{13}(y)= \int_{\mr} e^{i( \eta_1 y_1+ y_3\eta_3)} a_1(\eta_1,y_2, 0) a_3( \eta_3,0, y_2)  \ d\eta_1d\eta_3, \\ \mce_{13} \in  H_{\loc}^{0-, -m+\ha,\infty,-m-\ha}(\Omega)+ H_{\loc}^{0-, -m-\ha,\infty,-m+\ha}(\Omega), \\
 w_{23}(y)= \int_{\mr} e^{i( \eta_2 y_2+ y_3\eta_3)} a_2(\eta_2,y_1,0) a_3(\eta_3, y_1,0)  \ d\eta_2d\eta_3, \\ \mce_{23} \in  H_{\loc}^{0-, \infty,-m+\ha,-m-\ha}(\Omega)+ H_{\loc}^{0-, \infty,-m-\ha,-m+\ha}(\Omega).
 \end{gathered}\label{prod-21}
 \end{gather}
 
 The product of three distributions satisfies
 \begin{gather}
\begin{gathered}
v_1v_2v_3 = w+ \mce, \text{ where } \\
w= \int_{\mr^3} e^{i(y_1\eta_1+y_2\eta_2+y_3\eta_3)} a_1(\eta_1,0, 0) a_2( \eta_2,0, 0)  a_3(\eta_3,0,0) \ d\eta_1d\eta_2d\eta_3,  \\
\mce \in  H_{\loc}^{0, -m+\ha,-m-\ha,-m-\ha}(\Omega) + H_{\loc}^{0, -m-\ha,-m+\ha,-m-\ha}(\Omega) + H_{\loc}^{0, -m-\ha,-m-\ha,-m+\ha}(\Omega)
\end{gathered}\label{prod-3}
\end{gather}
\end{prop}
\begin{proof}  If $a_j$ is the principal symbol of $v_j,$  $j=1,2,$ then by definition  and in view of Proposition \ref{inc-con},
\begin{gather*}
v_1(y)= v_1'(y)+r_1(y), \ v_1'(y)= \int_{\mr} e^{i \eta_1 y_1} a_1(\eta_1, y_2,y_3) \; d\eta_1  + r_1, \\ r_1\in I^{m-1-\frac{n}{4}+\ha}(\Omega,\Sigma_1)\subset  H^{0-, -m+\ha,\infty,\infty}, \\
v_2(y)= v_2'(y)+r_2(y), \ v_2'(y)=\int_{\mr} e^{i \eta_2 y_2} a_2( \eta_2,y_1,y_3) \; d\eta_2  + r_2, \\ r_2\in I^{m-1-\frac{n}{4}+\ha}(\Omega,\Sigma_2)\subset  H^{0,\infty,-m+\ha,\infty}.
\end{gather*}
We know from Proposition \ref{inc-con} that $v_1' \in H^{0-, -m-\ha,\infty,\infty}$ and  $v_2' \in H^{0-, \infty, -m-\ha,\infty}$so it follows from 
Proposition \ref{BSP2} that 
\begin{gather*}
v_1' r_2\in H_{\loc}^{0-, -m-\ha, -m+\ha, \infty}(\Omega), \; v_2' r_1\in H_{\loc}^{0, -m+\ha, -m-\ha,\infty}(\Omega) \text{  and }
r_1r_2 \in  H_{\loc}^{0, -m+\ha, -m+\ha, \infty}(\Omega).
\end{gather*}

   To analyze the product $v_1'v_2'$ we expand the symbols in Taylor series in $y_2$ and $y_1$ respectively
\begin{gather*}
a_1(\eta_1,y_2,y_3)= a_1(\eta_1,0, y_3) + y_2 b_1(\eta_1, y_2,y_3), \;  
b_1\in  S^{m}(\mr \times \mr^2) \\
a_2(\eta_2, y_1,y_3)= a_2(\eta_2,0, y_3) + y_1 b_2(\eta_2, y_1,y_3), \ b_2 \in  S^{m}(\mr \times \mr^2).
\end{gather*} 
And hence
\begin{gather*}
a_1(\eta_1,y_2,y_3) a_2(\eta_2,y_1,y_3)= \\ (a_1(\eta_1, 0, y_3) + y_2 b_1(\eta_1,y_2,y_3)) (a_2(\eta_2, 0, y_3) + y_1 b_2(\eta_2,y_1,y_3))= \\
a_1(\eta_1,0, y_3)a_2(\eta_2,0, y_3) + y_1 a_1(\eta_1,0, y_3,)b_2(\eta_2,y_1,y_3)+ \\
y_2 b_1(\eta_1,y_2,y_3) a_2(\eta_2,0, y_3)+ y_1 y_2 b_1(\eta_1,y_2,y_3) b_2(y_2,y_3,\eta_2).
\end{gather*}
Now \eqref{prod-21} follows from Proposition \ref{reduction}.   The proof of \eqref{prod-3} follows from the same arguments.
\end{proof}

 \section{ The Regularity up to the Point of Triple Interaction}  In this section we analyze the behavior of the solution $u$ satisfying the hypotheses of Theorem \ref{triple} up to a small neighborhood of the point of triple interaction.
 
  By assumption, $0\in \{t=0\}$ and we first discuss the regularity of the solution $u$ to \eqref{Weq} in $\{t<T\},$ for any $T<0.$   Theorem \ref{REGSD} shows that the solution to \eqref{Weq} is conormal to the three hyperfurfaces in 
  $\{t<0\}$  and so according to Proposition \ref{coord-inv}, we can freely work with Beals spaces in local coordinates.
 \begin{prop}\label{reg-DI}    Let $u$  satisfy \eqref{Weq} with initial data $v=v_1+v_2+v_3,$ $v_j\in I^{m-\oq}(\Omega,\Sigma_j),$ and $m<-\fha.$  Then, for any $T<0,$ and $\Omega^T= \Omega\cap \{t<T\},$
 \begin{gather}
 \begin{gathered}
 u-v = w \in  H_{\loc}^{1-, -m-\ha,-m-\ha,-m-\ha}(\Omega^T) \text{ and }
 w \in \sum_{j,k=1}^3 I H_{\loc}^{-m+\ha-}(\Omega^T, \mcv_{jk})
  \end{gathered}\label{reg-DI1}
 \end{gather}
 \end{prop}
 \begin{proof}    Let $\Omega_1\Subset \Omega$ be a relatively compact  bicharacteristically convex subset of $\Omega$ and assume that $\supp \zed\Subset \Omega_1.$  It is enough to show that for any $T<0,$ the result holds in $\Omega_1^T=\Omega_1\cap \{t<T\}$ and for that we need to  build a finite open cover of $\overline{\Omega_1^T},$ consisting of open sets $U_r$ contained in $\Omega^{T+\del}$ and $T+\del<0,$  in the following way:  
 \begin{gather*}
   U_r, \ r \in \mck_{j} \subset \mn,  \text{ cover } \Sigma_j \cap \overline{\Omega_1^T}, \text{ and are disjoint from the other two hypersurfaces, }  \\
 U_r, \, r\in \mck_{jk}\subset \mn,  \text{ cover  } (\Sigma_j \cap \Sigma_k) \cap \overline{\Omega_1^T}, 
 \text{ and are disjoint from the other  double intersections,} \\
 U_r,  \, r\in \mck\in \mn,  \text{ cover } \overline{\Omega_1^T} \setminus \{ U_r: r\in \mck_{j}\cup \mck_{jk}, 1\leq j,k \leq 3\}.
 \end{gather*} 

  Let $\{0, 1, 2, \ldots, N\}$ be an ordering of the index sets $\mck \cup \mck_j \cup \mck_{jk}$ above and let $\{\chi_j,  0\leq j \leq N\}$ be a partition of unity subordinate  to the open cover $\{U_j, 0\leq j\leq N\}.$  Since $\Omega_1$ is bicharacteristically convex,  we have $u-v= E_+(P(y,u))$ in $\Omega_1^T$ and we can write
  \begin{gather}
  \begin{gathered}
    u-v = \sum_{k,l=0}^N  E_+(\chi_k P(y, \chi_l u)), \text{ in } \Omega_1^T, \text{ and also as } \\
  u-v = \sum_{j,k,l=0}^N \chi_j E_+(\chi_k P(y, \chi_l u)), \text{ in } \Omega_1^T.
  \end{gathered}\label{sum-pou}
  \end{gather}

  We know from Theorem \ref{REGSD} that
 \begin{gather}
 u \in \sum_{j,k=1}^3I H^{-m-\ha-}_{\loc} (\Omega_1^T, \mcv_{jk}), \label{reg-DI2}
 \end{gather} 
  and we will need to use the following results which are due to Bony \cite{Bony3,Bony4} and were used in the proof of Theorem \ref{REGSD}
  \begin{gather}
  \begin{gathered}
  \text{ if } s>\frac32, \;  I H_{\comp}^{s}( \Omega, \mcv_{\bullet}) \text{ is an algebra, } \bullet=j \text{ or } jk,  j,k=1,2,3,\\
  E_+:  I H_{\comp}^{s}( \Omega, \mcv_{j}) \longrightarrow I H_{\loc}^{s+1}( \Omega, \mcv_{j}), \;\ j=1,2,3, \\
   E_+:  I H_{\comp}^{s}( \Omega, \mcv_{jk}) \longrightarrow I H_{\loc}^{s+1}( \Omega, \mcv_{jk}), \;\ j,k=1,2,3.
    \end{gathered}\label{PCN}
    \end{gather}
 
 Since $0\not \in \Omega_1^T,$ it follows from the first part of \eqref{PCN} that $P(y,u)\in \sum_{j,k=1}^3I H^{-m-\ha-}_{\loc} (\Omega_1^T, \mcv_{jk}),$ and hence $w=u-v=E_+(P(y,u))\in \sum_{j,k=1}^3I H^{-m+\ha-}_{\loc} (\Omega_1^T, \mcv_{jk}).$ This proves half of \eqref{reg-DI1}.
  
  Suppose that $l\in \mck,$ hence  $\supp \chi_l$ does not intersect any of the hypersurfaces, then by \eqref{reg-DI2}, 
  $P(y,\chi_l u)\in \CI,$   and so $E_+\left( \chi_k P(y,\chi_l u)\right)\in \CI(\Omega)$ for any $1\leq k \leq N.$ 
 
  Suppose $l\in \mck_1,$ so the support of $\chi_l$ intersects $\Sigma_1$ but does not intersect  $\Sigma_2\cup \Sigma_3.$  Then by \eqref{reg-DI2}, $\chi_l u\in I H^{-m-\ha-}( U, \mcv_{1}).$ Since $m<-\fha,$
  $I H^{-m-\ha-}( U, \mcv_{1})$ is an algebra and hence   $P(y,\chi_l u)\in I H^{-m-\ha-}( U, \mcv_{1})$ and so from \eqref{PCN}, $ E_+\left(\chi_k  P(y,\chi_l u)\right)\in I H^{-m+\ha-}( \Omega, \mcv_{1}),$ for any $1\leq k \leq N.$
 
 Since $\mcv_1$ is spanned by  $\{y_1\p_{y_1}, \p_{y_2}, \p_{y_3}\},$  we conclude that
  \begin{gather}
E_+(\chi_k P(y,\chi_l u))  \in H^{-m+\ha, 0, \infty, \infty}_{\loc}(\Omega) \subset H^{1-, -m-\ha, \infty, \infty}_{\loc}(\Omega), \ k\in \mck_1, \  1\leq j \leq N. \label{regu-SI}
 \end{gather}
 
 Similarly, we obtain
 \begin{gather}
 \begin{gathered}
E_+(\chi_k P(y,\chi_l u))   \in H^{1-, \infty, -m-\ha, \infty}_{\loc}(\Omega), \text{ if } l\in \mck_2, \, 1\leq k \leq N,  \\
 E_+(\chi_k P(y,\chi_l u))  \in H^{1-, \infty, \infty, -m-\ha}_{\loc}(\Omega), \text{ if } l\in \mck_3, \, 1\leq k \leq N. 
  \end{gathered}\label{regu-SI1}
  \end{gather}
  
  We have concluded the following:  
  \begin{gather}
  \begin{gathered}
   \sum_{j=1}^N \sum_{k\in \mck \cup \mck_1 \cup \mck_2 \cup \mck_3}  E_+(\chi_j P(y,\chi_k u))  \in H^{1-, -m-\ha, -m-\ha, -m-\ha}_{\loc}(\Omega).
 \end{gathered}\label{outdi}
\end{gather}

Suppose that $l\in \mck_{12}.$  Then equation \eqref{reg-DI2} implies that $\chi_l u \in I H^{-m-\ha-}_{\loc}(\Omega, \mcv_{12}).$ Since $m<-\fha,$ $P(y,\chi_l u)\in I H_{\loc}^{-m-\ha}(\Omega, \mcv_{12}).$   If $k\not\in \mck_{12},$ we have the following possibilities: 
\begin{enumerate}[1.]
\item If $k \in \mck \cup \mck_3 \cup \mck_{13} \cup \mck_{23},$ then  $\chi_k P(y,\chi_l u) \in C^\infty(\Omega)$ and therefore $E_+(\chi_k P(y,\chi_l u))\in C^\infty(\Omega).$ 
\item If  $k \in \mck_1,$ then   $\chi_ k P(y,\chi_l u)\in I H_{\loc}^{-m-\ha-}(\Omega, \mcv_{1}),$ and so
$E_+(\chi_k P(y,\chi_l u)\in H_{\loc}^{1-, \frac{m}{2}-\ha, \infty, \infty}(\Omega).$ 
\item If $k \in \mck_2,$ then   $\chi_k P(y,\chi_l u)\in I H_{\loc}^{-m-\ha-}(\Omega, \mcv_{2}),$ and so 
$E_+(\chi_k P(y,\chi_l u)\in H_{\loc}^{1-, \infty, \frac{m}{2}-\ha, \infty, }(\Omega).$
\end{enumerate}

If $k,l\in \mck_{12},$ then  we use the second equation of \eqref{sum-pou}.  In this case  
$E_+(\chi_k P(y,\chi_l u))\in I H_{\loc}^{-m+\ha-}(\Omega, \mcv_{12})$ and so $\chi_j E_+(\chi_k P(y,\chi_l u))$ satisfies one of the following:
\begin{enumerate}[1.]
\item $\chi_j E_+(\chi_k P(y,\chi_l u)) \in \CI(\Omega),$ if $k \in \mck \cup \mck_3  \cup \mck_{13}\cup \mck_{23}.$
\item   If $j\in \mck_1$ or $j\in \mck_2,$ or
$\chi_j E_+(\chi_k P(y,\chi_l u))\in I H_{\loc}^{-m+\ha-}(\Omega, \mcv_{j}),$ $j=1$ or $j=2.$ This implies that  $\zeta E_+(\chi_j P(y,\chi_k u)) \in H_{\loc}^{1-, -m-\ha, -m-\ha,\infty}(\Omega).$
\end{enumerate}
We use the same argument for the other two double intersections and we conclude the following: In $\Omega_1^T,$
\begin{gather}
\begin{gathered}
\text{ if  }  \mca(u) =\sum_{j,k,l \in \mck_{12}} \chi_j E_+(\chi_k P(y,\chi_l u))+ 
\sum_{j,k,l \in \mck_{13}} \chi_j E_+(\chi_k P(y,\chi_l u))+\\
 \sum_{j,k,l \in  \mck_{23}} \chi_j E_+(\chi_k P(y,\chi_l u)), \text{ then } \\
u-v-\mca(u)=\mcr \in H^{1-,-m-\ha,-m-\ha,-m-\ha}_{\loc}(\Omega_1^T), 
\end{gathered}\label{diag-sum}
\end{gather}

Now we are left to show that $\mca \in H_{\loc}^{1-, -\frac{m}{2}-\ha, -\frac{m}{2}-\ha, -\frac{m}{2}-\ha}(\Omega_1^T).$
We will show that
\begin{gather}
\sum_{j,k,l\in \mck_{12}}  \chi_j E_+(\chi_k P(y,\chi_l u))  \in H^{1-, -m-\ha, -m-\ha, -m-\ha}_{\loc}(\Omega)\label{outdi1}
\end{gather}
and other  terms corresponding to $\mck_{13}$ and $\mck_{23}$ can be handled identically. 

 In view of the cases already analyzed,  we may assume that 
  \begin{gather*}
\text{ if } j \in \mck_{12}, U_j \subset \Upsilon_{\del}=\{(y_1,y_2, y_3) \in \Omega_1^T \setminus \mco : |(y_1,y_2)|<\del \},
 \end{gather*}
 with $\del$ to be chosen.

By the discussion above, without loss of generality, we may assume that the open sets $U_j$ if $j\in \mck_{12}$ are disjoint from $U_j,$ $j\in \mck_{13} \cup \mck_{23},$ so if $r\in \mck_{12},$ then in view of \eqref{diag-sum},
\begin{gather}
\chi_r(u-v-\mcr)= \sum_{j,k,l\in \mck_{12}} \chi_r\chi_j E_+( \chi_k P(y,\chi_l u)). \label{juv}
\end{gather}
 
We know from equation \eqref{reg-DI2} that  $\chi_l u \in  I H_{\loc}^{-m-\ha-}( \Omega, \mcv_{12}),$ and since this space is an algebra, $\chi_k P(y,\chi_l u)\in   I H^{-m-\ha-}( \Omega, \mcv_{12}).$  Therefore, by \eqref{PCN},  $\chi_j E_+(\chi_k P(y,\chi_r u))\in   I H^{-m+\ha-}( \Omega, \mcv_{12}).$  Since $\mcv_{12}$ is spanned by 
$\{ y_1\p_{y_1}, y_2\p_{y_2}, \p_{y_3}\},$ this implies that
\begin{gather*}
\chi_j E_+( \chi_k P(y,\chi_l u)) \in H_{\loc}^{-m+\ha-, 0,0, \infty} (\Omega) \subset  H_{\loc}^{0-, -\frac{m}{2}+\frac{1}{4}, -\frac{m}{2}+\oq, \infty}(\Omega), \ j,k,l\in \mck_{12}.
\end{gather*}

 If $r \in \mck_{12},$  \eqref{juv} implies that $\chi_r u\in H_{\loc}^{0-, -\frac{m}{2}+\frac{1}{4}, -\frac{m}{2}+\oq, \infty}(\Omega).$  Since $m<-\fha,$ this space is an algebra by Proposition \ref{BSP2} and so Proposition \ref{BSP4} gives that if $\del$ is small enough it follows from \eqref{juv} that 
\begin{gather*}
\chi_r(u-v-\mcr)=\sum_{k,m\in \mck_{12}} \chi_r \chi_j E_+( \chi_k P(y,\chi_l u)) \in  H_{\loc}^{1-,-\frac{m}{2}+\oq, -\frac{m}{2}+\oq, \infty}(\Omega) \subset  \\ H_{\loc}^{0-,-\frac{m}{2}+\frac34, -\frac{m}{2}+\frac34, \infty}(\Omega), \ r\in \mck_{12}.
\end{gather*}

 Therefore,  if $r\in \mck_{12},$ $\chi_r u \in H_{\loc}^{0-,-\frac{m}{2}+\frac34, -\frac{m}{2}+\frac34,\infty}(\Omega) $ and so, again using Proposition \ref{BSP2} and Proposition \ref{BSP4}, we find that, for small enough $\del,$
 \begin{gather*}
 \chi_j E_+( \chi_k P(y,\chi_l u)) \in H_{\loc}^{1-,-\frac{m}{2}+\frac34, -\frac{m}{2}+\frac34, \infty}(\Omega), \ j,k,l\in \mck_{12}, \text{ and } \\
\chi_r u \in H_{\loc}^{1-,-\frac{m}{2}+\frac34, -\frac{m}{2}+\frac34, \infty}(\Omega)\subset  H_{\loc}^{0-,-\frac{m}{2}+\frac74, -\frac{m}{2}+\frac74, \infty}(\Omega), \ r \in \mck_{12}.
\end{gather*}
After $M$   iterations, where $M$ is equal to the smallest integer  greater than or equal to $-m-\tha,$ 
we find that 
\begin{gather*}
\begin{gathered}
\chi_r u \in   H_{\loc}^{0-,-m-\frac12, -m-\frac12, \infty}(\Omega), \ r\in \mck_{12}.
\end{gathered}
\end{gather*}

 Applying Proposition \ref{BSP2} and Proposition \ref{BSP4} to $\chi_k P(y,\chi_l u)$ we find that ,
 \begin{gather*}
\chi_j E_+(\chi_k P(y,\chi_l u)) \in   H_{\loc}^{1-,-m-\frac12, -m-\frac12, \infty}(\Omega), \ j, k,l \in \mck_{12}.
\end{gather*}
We repeat the argument  for $\mck_{13}$ and $\mck_{23}$ and prove \eqref{diag-sum}.  

This shows that $u-v=w \in  H_{\loc}^{1-,-m-\frac12, -m-\frac12, -m-\ha}(\Omega_1^T)$ for any $T<0.$
This ends the proof of the proposition.
\end{proof}
Now we analyze the regularity of the solution $u$ to \eqref{Weq} in a neighborhood of $\{0\}.$ The goal is to show that in a suitable sense, the top order singularity of $u$ is still  given by $v.$ This will be used in the next section to show that the terms coming from the double interactions in $\Omega^-=\Omega\cap \{t<0\},$ as in for example Proposition \ref{prod-three-dist},   will not affect the principal symbol of the new singularities of $u$ on the cone.  

\begin{prop}\label{reg-DI3} Let $\Omega_2$ be a bicharacteristcally convex neighborhood  of $\{0\},$ which is small enough so that Proposition \ref{BSP3} holds in $\Omega_2.$ Let $T<0$ be small enough such that $\Omega_2\cap \{t>3T\}\not=\emptyset$ and let  $w$ be given by \eqref{reg-DI1}.  Let
$\chi\in C^\infty(\Omega),$ $\chi=1$ on $\{t> 2T\}$ and $\chi=0$ on $\{t<3T\}.$   Let $W$ satisfy $\square W=0$ in $\Omega_2,$ and $W=w$ for $t<T.$  Then 
\begin{gather}
\begin{gathered}
\chi u-\chi v- \chi W = E_+( \chi P(y, \chi u +(1-\chi)(v+w)))\in H^{1-, -m-\ha, -m-\ha,-m-\ha}_{\loc} (\Omega_2), \\
\text{ with } W \in \sum_{jk=1}^3 IH_{\loc}^{-m+\ha-}(\Omega_2,\mcv_{jk}) \text{ and } H^{1-, -m-\ha, -m-\ha,-m-\ha}_{\loc} (\Omega_2).
\end{gathered}\label{NecReg}
\end{gather}
\end{prop}
\begin{proof}   We know from equation \eqref{reg-DI1} that 
\begin{gather*}
w\in \sum_{jk=1}^3 I H^{-m+\ha-}_{\loc}(\Omega^T,\mcv_{jk}), \text{ and } w\in H_{\loc}^{1-,-m-\frac12, -m-\frac12, -m-\ha}(\Omega^T).
\end{gather*} 
Therefore,
\begin{gather*} 
\square \chi W= [\square, \chi] w \in H_{\loc}^{0-,-m-\ha, -m-\ha, -m-\ha}(\Omega_2^T)\cap \left(\sum_{jk=1}^3I H^{-m-\ha-}_{\loc}(\Omega_2^T,\mcv_{jk})\right),\\
\chi W=0, \ t<3T,
\end{gather*}
and since  $0\not \in [\square,\chi]w,$  it follows from \eqref{PCN} that 
$W\in \sum_{jk=1}^3I H^{-m+\ha-}_{\loc}(\Omega_2,\mcv_{jk}),$ and for $\Omega_2$ small enough, it follows from Proposition \ref{BSP3} that
$W \in H_{\loc}^{1-,-m-\frac12, -m-\frac12, -m-\ha}(\Omega_2).$

We have, for $T<0$ small,
\begin{gather*}
\square u= P(y,u) \text{ in } \Omega_2,\\
u=v+w, \text{ for } t<T.
\end{gather*}

Since by assumption $\square v=0,$ it follows that
\begin{gather}
\begin{gathered}
\square \chi( u-v-W)= \chi P(y,u),\\
\chi(u-v-w)=0, \text{ for } t<T.
\end{gathered}\label{cut-off-u}
\end{gather}
and hence
\begin{gather*}
\chi u=\chi(v+W)+ E_+( \chi P(y, u))= \chi(v+W)+ E_+( \chi P(y, \chi u+(1-\chi)u))=\\ \chi(v+W)+ E_+( \chi P(y, \chi u +(1-\chi)(v+w))).
\end{gather*}

Now we bootstrap as in the argument used in the proof of Proposition \ref{reg-DI}.  We know that $u\in H^{-m-\ha}_{\loc}(\Omega),$ so in particular, $u \in H^{0-, -\frac{m}{3}-\frac{1}{6}, -\frac{m}{3}-\frac{1}{6},-\frac{m}{3}-\frac{1}{6}}_{\loc} (\Omega_2).$ Since $m<-\fha,$ this space is an algebra, and since $v, (1-\chi)w,W\in H_{\loc}^{0-,-m-\frac12, -m-\frac12, -m-\ha}(\Omega_2),$ it follows from Proposition \ref{BSP3} that for small $\Omega_2,$
\begin{gather*}
\chi (u-v) \in H^{1-, -\frac{m}{3}-\frac{1}{6}, -\frac{m}{3}-\frac{1}{6},-\frac{m}{3}-\frac{1}{6}}_{\loc} (\Omega_2)\subset
 H^{0-, -\frac{m}{3}+\frac{1}{6}, -\frac{m}{3}+\frac{1}{6},-\frac{m}{3}+\frac{1}{6}}_{\loc} (\Omega_2).
 \end{gather*}
 We repeat the argument a finite number of times, and find that
 \begin{gather*}
\chi u \in H^{0-, -m-\ha, -m-\ha,-m-\ha}_{\loc} (\Omega_2).
\end{gather*}
Now we use \eqref{cut-off-u} and Proposition \ref{BSP3} to conclude that $E_+( \chi P(y, \chi u +(1-\chi)(v+w)))\in H^{1-, -m-\ha, -m-\ha,-m-\ha}_{\loc} (\Omega_2).$ This concludes the proof of the proposition.
\end{proof}

\begin{remark}\label{REM1}
 We have  discussed the singularities of the solution $u$  of \eqref{Weq} in terms of Beals spaces  in $\Omega^T,$ for $T<0,$  in Proposition \ref{reg-DI} and because of Proposition \ref{coord-inv}, this does not depend on the choice of local coordinates.  But the value of $u$ on $\{t<T\},$ $T<0$ small,  is used as initial data in the analysis of  the regularity of the solution in $\Omega_2$ given by Proposition \ref{reg-DI3}. This shows that \eqref{NecReg} holds independently of  the choice of coordinates in $\Omega_2$  used to define the corresponding Beals spaces.
 \end{remark}

\section{ Singularities generated by the triple interaction}\label{gen-sing}

We know from Theorem \ref{REG}, that microlocally away from the surfaces $\Sigma_j$ the solution $u$ to \eqref{Weq} with conormal initial data $v_j,$ $j=1,2,3,$ is a conormal distribution to $\mcq,$ but  with a $H^s$-based symbol.   In this section we compute the principal symbol of $u$ microlocally near  $\La_\mcq$ and away from $\Sigma_j$ and their intersection with $\mcq.$
 \subsection{The Computation of the Principal Symbol on $\mcq\setminus (\Sigma_1\cup \Sigma_2\cup \Sigma_3)$}\label{comp-PS}

Now we come to our main technical result:
\begin{theorem}\label{triple1}    Let $\Omega,$ $\zed$ and $P(y,u),$  $\square,$ $\Sigma_j,$ $j=1,2,3$ and $\mcq$ satisfy the hypotheses of Theorem \ref{triple}.  Let $y=(y_1,y_2,y_3)$ be local coordinates valid in $\Omega$  such that $\Sigma_j=\{y_j=0\}.$ Suppose that  $v=v_1+v_2+v_3,$ is such that $v_j\in I^{m-\oq}(\Omega,\Sigma_j),$  $m<-\fha,$  is  elliptic and that
$a_1(\eta_1,y_2,y_3), a_2(\eta_2,y_1,y_3), a_3(\eta_3,y_1,y_2)\in S^m(\mr\times \mr^2),$  are the principal symbols of  $v_1,$ $v_2$ and $v_3$ respectively.    Let $u\in H^{-m-\ha-}_{\loc}(\Omega)$ satisfy \eqref{Weq}. 
Then  in a small neighborhood of $\{0\}$ and microlocally in the region where $\lan \eta_j\ran \geqs \lan \eta\ran,$ $j=1,2,3,$ 
\begin{gather}
\begin{gathered}
u(y)= u_0(y)+u_1(y), \text{ where } \\
 u_0= (\p_u^3 P)(0, u(0)) E_+\left( V \right)  \in H^{-3m-\ha-}\setminus H^{-3m-\ha}, \\
 V(y)= \int_{\mr^3} e^{i(y_1\eta_1+y_2\eta_2+y_3\eta_3)} a_1(\eta_1,0,0) a_2(\eta_2,0,0) a_3(\eta_3,0,0) \ d\eta_1 d\eta_2 d\eta_3  \\ \text{ and } 
 u_1\in H^{-3m-\ha+r/2}, \text{ provided }   r\in (0, 1-\frac{2}{-m-\ha}),
\end{gathered}\label{reguL}
\end{gather}
where as before, $E_+$ denotes the forward fundamental solution to $\square.$  
\end{theorem}
\begin{proof}   We know that for a domain $\Omega_2$ as Proposition \ref{reg-DI3}, equation  \eqref{NecReg} holds and we have in $\Omega_2,$ 
\begin{gather}
\begin{gathered}
\chi u-\chi(v+W)= E_+( \chi P(y, \chi u +(1-\chi)(v+w)))  \in H_{\loc}^{1-,-m-\ha, -m-\ha, -m-\ha}(\Omega_2).
\end{gathered} \label{New-Regu-u}
\end{gather}
As discussed in Remark \ref{REM1},  equation \eqref{NecReg} is independent of the choice of coordinates  used in in $\Omega_2$ to define the Beals spaces.

Next we  appeal to Proposition \ref{pir-dec}  and we  write
\begin{gather*}
v= \nu+ \mce, \text{ where }\mce \in C^\infty \text{ and }
 \nu= \nu_1+\nu_2+\nu_3, \;\ \nu_j\in \ido{}^{m-\oq}(\Omega,\Sigma_j). 
\end{gather*}

In what follows, to simplify the notation, we denote
\begin{gather}
\mcw= \mce + E_+(  \chi P(y,\chi u+ (1-\chi)((v+w))) \in H_{\loc}^{1-,-m-\ha, -m-\ha, -m-\ha}(\Omega_2), \label{defmcw}
\end{gather}
where we have used \eqref{NecReg}. Notice that, since $\nu(0)=0,$
\begin{gather}
\mcw(0)=u(0). \label{wof0}
\end{gather}

  We  iterate this formula and  deduce that 
\begin{gather*}
\chi (u-v-W) = E_+( \chi P(y, \nu+\mcw)), \text{ in } \Omega_2.
\end{gather*}

Since $P(y,u)$ is a polynomial of degree $N,$ 
\begin{gather*}
\chi(u-v-W)= E_+ ( \chi P(y, \nu+\mcw))=  E_+\left( \chi \sum_{j=0}^N \frac{1}{j!} (\p_u^{j} P)(y,\mcw) \nu^j\right).
\end{gather*}
Since $(\p_u^{j} P)(y,\mcw)$ is a polynomial of degree $N-j$ in  $\mcw,$  then in virtue of Proposition \ref{BSP2} and equation \eqref{NecReg},
\begin{gather}
(\p_u^{j}P)(y,\mcw) \in  H_{\loc}^{1-,-m-\frac12, -m-\frac12, -m-\frac12}(\Omega_2)\subset \mci^{0,m}(\Omega_2), \label{NEED}
\end{gather}
where to simplify the notation used below, we define the following spaces:
\begin{gather*}
\mci^{0,m}(\Omega_2)= H_{\loc}^{0-, -m+\ha,-m-\ha,-m-\ha}(\Omega_2)\cap H_{\loc}^{0-, -m-\ha,-m+\ha,-m-\ha}(\Omega_2)\cap H_{\loc}^{0-, -m-\ha,-m-\ha,-m+\ha}(\Omega_2), \\
\mch^{0,m}(\Omega_2)= H_{\loc}^{0-, -m+\ha,-m-\ha,-m-\ha}(\Omega_2)+ H_{\loc}^{0-, -m-\ha,-m+\ha,-m-\ha}(\Omega_2)+ H_{\loc}^{0-, -m-\ha,-m-\ha,-m+\ha}(\Omega_2), \\
\mch^{1,m}(\Omega_2)= H_{\loc}^{1-, -m+\ha,-m-\ha,-m-\ha}(\Omega_2)+ H_{\loc}^{1-, -m-\ha,-m+\ha,-m-\ha}(\Omega_2)+ H_{\loc}^{1-, -m-\ha,-m-\ha,-m+\ha}(\Omega_2). 
\end{gather*}
The following result separates the terms with higher order of regularity of the solution $u$ to \eqref{Weq}:
\begin{lemma}\label{Reg-Terms} Let $u,$ $\mcw$ and $P(y,u)$ be as above and let $V$ be defined in \eqref{reguL}. Then
\begin{gather}
\begin{gathered}
\chi(u-v-W)-E_+\left( \chi (\p_u^{3} P)(y,\mcw) V\right)
 \in \mch^{1,m}(\Omega_2).
\end{gathered}\label{Reg-terms1}
\end{gather}
\end{lemma}
\begin{proof} We know that
\begin{gather*}
\nu^j=(\nu_1+\nu_2+\nu_3)^j= \sum_{|\alpha|=j} C_\alpha \nu_1^{\alpha_1} \nu_2^{\alpha_2} \nu_3^{\alpha_3}, \;\ \alpha=(\alpha_1,\alpha_2,\alpha_3), \text{ and we split this sum as} \\
\nu^2= \sum_{k=1}^3 \nu_k^2+\sum_{|\alpha|=2, \alpha_k,\alpha_i\not=0} C_\alpha \nu_i^{\alpha_i} \nu_k^{\alpha_k},  \text{ and in general for } j\geq 3,  \\
\nu^j =  \sum_{k=1}^3 \nu_k^j+ \sum_{i,k=1}^3 \sum_{|\alpha|=j, \alpha_i,\alpha_k\not=0} C_\alpha \nu_i^{\alpha_i} \nu_k^{\alpha_k} +  \sum_{|\alpha|=j, \alpha_1,\alpha_2,\alpha_3\not=0} C_\alpha \nu_1^{\alpha_1}\nu_2^{\alpha_2} \nu_3^{\alpha_3}.
\end{gather*}
Therefore we write
\begin{gather*}
\chi(u- v-W) = E_+\left( \chi P(y,\mcw)\right)+  E_+ \left(  \chi (\p_u P)(y,\mcw) \nu\right) + E_+( \mcl) \text{ where } \\
\mcl=  \sum_{j=2}^N \frac{1}{j!} \chi (\p_u^{j} P)(y, \mcw) \left( \sum_{k=1}^3  \nu_k^j+  \sum_{i,k=1}^3 \sum_{|\alpha|=j, \alpha_i,\alpha_k\not=0} C_\alpha \nu_i^{\alpha_i} \nu_k^{\alpha_k} \right) +\\ 
\chi (\p_u^3 P)(y,\mcw) \nu_1\nu_2\nu_3+ 
\sum_{j=4}^N \frac{1}{j!} \chi (\p_u^{j} P)(y,\mcw)  \left( \sum_{|\alpha|=j, \alpha_1,\alpha_2,\alpha_3\not=0} C_\alpha \nu_1^{\alpha_1}\nu_2^{\alpha_2} \nu_3^{\alpha_3}\right).
\end{gather*}

Notice that the coefficient of $\nu_1\nu_2\nu_3$ is equal to $3!.$
To control products involving  one or two factors $\nu_i^{\alpha_j}$ or $\nu_i^{\alpha_i}\nu_j^{\alpha_j},$ we proceed as in \cite{Beals4}. It follows from Proposition \ref{BSP2} that for $\alpha_j\geq 1,$
\begin{gather*}
\nu_1^{\alpha_1}  \in H_{\loc}^{0-, -m-\ha, \infty, \infty}(\Omega_2), \;\ \nu_2^{\alpha_2}  \in H_{\loc}^{0-, \infty, -m-\ha, \infty}(\Omega_2), \\
 \nu_3^{\alpha_3}  \in H_{\loc}^{0-, \infty, \infty, -m-\ha}(\Omega_2), \;\ 
\nu_1^{\alpha_1} \nu_2^{\alpha_2} \in H_{\loc}^{0-, -m-\ha, -m-\ha, \infty}(\Omega_2), \\ 
\nu_1^{\alpha_1} \nu_3^{\alpha_3} \in H_{\loc}^{0-, -m-\ha, \infty, -m-\ha}(\Omega_2), \;\ 
\nu_2^{\alpha_2} \nu_3^{\alpha_3} \in H_{\loc}^{0-, \infty, -m-\ha, -m-\ha}(\Omega_2).
\end{gather*}
We know from \eqref{defmcw} and  Proposition \ref{BSP2}  that for any $k\in \mn,$
\begin{gather}
\begin{gathered}
\mcw^k \in  H_{\loc}^{1-, -m-\ha, -m-\ha, -m-\ha}(\Omega_2) \text{ and hence } 
\mcw^k \in \mci^{0,m}(\Omega_2),
\end{gathered}\label{reg-W}
\end{gather}
 and therefore, for any $k,$ and $\alpha_j\geq 1,$
 \begin{gather}
 \begin{gathered}
\nu_1^{\alpha_1}  \mcw^k \in  H_{\loc}^{0-, -m-\ha, -m-\ha, -m+\ha}(\Omega_2), \;\ 
\nu_2^{\alpha_2}  \mcw^k \in  H_{\loc}^{0-, -m+\ha, -m-\ha, -m-\ha}(\Omega_2), \\
\nu_3^{\alpha_3}  \mcw^k \in  H_{\loc}^{0-, -m+\ha, -m-\ha, -m-\ha}(\Omega_2), \;\
\nu_2^{\alpha_2} \nu_3^{\alpha_3} \mcw^k \in  H_{\loc}^{0-, -m+\frac12, -m-\ha, -m-\ha}(\Omega_2), \\
\nu_1^{\alpha_1} \nu_3^{\alpha_3} \mcw^k \in  H_{\loc}^{0-, -m-\frac12, -m+\frac12, -m-\ha}(\Omega_2), \;\
\nu_1^{\alpha_1} \nu_2^{\alpha_2} \mcw^k \in  H_{\loc}^{0-, -m-\frac12, -m-\ha, -m+\frac12}(\Omega_2),
\end{gathered}\label{singleP}
\end{gather}
and so we conclude that
\begin{gather}
\begin{gathered}
\chi(u-v-W) = E_+\left(T_1+T_2 +T_3+T_4\right) + E_+\left ( \chi \p_u^3 P(y,u) \nu_1\nu_2\nu_3\right)+\\ 
E_+\left[\chi \sum_{j=4}^N \frac{1}{j!} (\p_u^{j} P)(y,\mcw)  \left( \sum_{|\alpha|=j, \alpha_k\not=0} \nu_1^{\alpha_1}\nu_2^{\alpha_2} \nu_3^{\alpha_3}\right)\right], \\
T_1\in H_{\loc}^{1-,-m-\ha,-m-\ha,-m-\ha}(\Omega_2), \; T_2\in H_{\loc}^{0-, -m+\frac12, -m-\ha, -m-\ha}(\Omega_2), \\
 T_3 \in H_{\loc}^{0-, -m-\ha, -m+\frac12, -m-\ha}(\Omega_2) \text{ and } 
 T_4 \in H_{\loc}^{0-, -m-\ha, -m-\ha, -m+\frac12}(\Omega_2).
\end{gathered}\label{terms}
\end{gather}

If $N=3,$ then  $\chi (\p_u^3 P)(y,u)=6 \chi \zed(y)(y) a_3(y),$ and by expanding $\chi(y) \zed(y)a_3(y)$ in Taylor series about $\{y=0\},$ we see that  in a neighborhood of $\{0\}$ where $\chi=\zed=1,$ using Proposition \ref{reduction} and Proposition \ref{inc-con}, we find that
\begin{gather*}
(\chi \zed(y) a_3(y)-a_3(0))\nu_1\nu_2\nu_3 \in \mchz
\end{gather*}
Furthermore,  Proposition  \ref{prod-three-dist} and again Proposition \ref{inc-con} give that
\begin{gather}
\begin{gathered}
 a_3(0)\nu_1\nu_2\nu_3- a_3(0) V  \in \mchz,
\end{gathered}\label{proofn3}
\end{gather}
where $V$ is defined in \eqref{reguL}. This proves \eqref{Reg-terms1} in the case $N=3.$

Now we consider the case $N>3,$ and analyze the terms with  $j>3$ in \eqref{terms}.  In this case, $\alpha_j>1$ for at least one $j\in \{1,2,3\}.$ Let us say,  $\alpha_1>1,$ then in view of Proposition \ref{pir-prod1} $\nu_1^{\alpha_1}\in I^{m+(1-\alpha_1) k(m)-\oq}(\Omega_2, \Sigma_1)$  and  in particular
$\nu_1^{\alpha_1} \in H_{\loc}^{0-, -m +(\alpha_1-1)k(m) -\ha, \infty, \infty}(\Omega_2).$ Since $m<-\fha,$ it follows that $k(m)\geq 1$ and since $\alpha_1\geq 2,$ 
$(\alpha_1-1)k(m)\geq 1$ and we find that $\nu_1^{\alpha_1} \in H_{\loc}^{0-, -m +\ha, \infty, \infty}(\Omega_2).$ So, using \eqref{reg-W} and Proposition \ref{BSP2},  we conclude that
$\nu_1^{\alpha_1} \nu_2^{\alpha_2} \nu_3^{\alpha_3} \in H_{\loc}^{0-, -m+\ha, -m-\ha, -m-\ha}(\Omega_2)$  and that
\begin{gather*}
\nu_1^{\alpha_1} \nu_2^{\alpha_2} \nu_3^{\alpha_3}  (\p_u^{j} P)(y,\mcw)  \in H_{\loc}^{0-, -m+\ha, -m-\ha, -m-\ha}(\Omega_2).
\end{gather*}
Using the same argument when $\alpha_2>1$ or $\alpha_3>1$ we conclude that
\begin{gather}
\begin{gathered}
\sum_{j=4}^N \frac{1}{j!} (\p_u^{j} P)(y,\mcw)  \left( \sum_{|\alpha|=j, \alpha_1,\alpha_2,\alpha_3\not=0} \nu_1^{\alpha_1}\nu_2^{\alpha_2} \nu_3^{\alpha_3}\right) \in \mchz.
\end{gathered}\label{NEW-ReG}
\end{gather}
 Therefore we  deduce from \eqref{NEW-ReG}, \eqref{terms} and Proposition \ref{BSP3} that
 \begin{gather}
  \begin{gathered}
\chi( u-v-W) -E_+\left[ \chi (\p_u^3 P)(y,\mcw)  \left(\nu_1\nu_2 \nu_3\right)\right] \in \mcho
  \end{gathered}\label{NEW-ReG1}
  \end{gather}
As in the case $N=3,$ since $\chi=1$ in a neighborhood of $\{0\},$  Proposition \ref{reduction} and the fact that
$(\p_u^3P)(y,\mcw)\in H_{\loc}^{1-,-m-\ha,-m-\ha,-m-\ha}$ give that
\begin{gather*}
E_+\left[ \left(\chi(\p_u^3 P)(y,\mcw)- (\p_u^3 P)(y,\mcw)\right) \left(\nu_1\nu_2 \nu_3\right)\right]\in \mcho 
\end{gather*}
But  from Proposition \ref{prod-three-dist} we find that if $V$ is given by \eqref{reguL}, then
\begin{gather}
\begin{gathered}
(\p_u^3 P)(y,\mcw)  \left( \nu_1\nu_2 \nu_3\right)-  (\p_u^3 P)(y,\mcw) V \in \mchz. 
 \end{gathered} \label{NEW-ReG2}
  \end{gather}
 Finally, \eqref{NEW-ReG2}, \eqref{NEW-ReG1} and Proposition \ref{BSP3} imply \eqref{Reg-terms1} and this ends the proof of Lemma \ref{Reg-Terms}.
  \end{proof}

Microlocally in the region where $\lan\eta_j\ran \geqs \lan \eta\ran,$ $j=1,2,3,$
\begin{gather*}
\mchz \subset H^{-3m-\frac12-}(\Omega_2) \text{ and } \mcho \subset H^{-3m+\frac12-}(\Omega_2).
 \end{gather*}

 Therefore, provided  $(\p_u^{3} P)(y,\mcw)\not=0$ near $\{y=0\},$ we deduce from  \eqref{Reg-terms1} that microlocally in the region where $\lan\eta_j\ran \geqs \lan \eta\ran,$ $j=1,2,3,$ $v, W\in H^\infty,$ 
 \begin{gather*}
 E_+\left((\p_u^{3} P)(y,\mcw)\right) \left(\nu_1\nu_2\nu_3\right) \in H_{\loc}^{-3m-\ha-} \text{ and } 
u-E_+\left( (\p_u^{3} P)(y,\mcw) V\right) \in H_{\loc}^{-3m+\frac12-}, 
 \end{gather*}
where $V$ is defined in \eqref{reguL}. We conclude that $E_+\left((\p_u^{3} P)(y,\mcw) V\right)$  is the term of top singularity of $u$ in the region where $\eta_j$ is elliptic, $j=1,2,3.$   However, this is not good enough to compute the principal symbol of $u$ near the light cone because the term $\left( (\p_u^{3} P)(y,\mcw)\right)$ depends on $u.$ One might suspect that in fact the top order singularity comes from
 \begin{gather}
\left( (\p_u^{3} P)(0,\mcw(0)) \right) V=\left( (\p_u^{3} P)(0,u(0)) \right) V. \label{auxEq1}
 \end{gather}
here we used that $u(0)=\mcw(0),$ see \eqref{wof0}. However, the term  $(\p_u^{3} P)(y,\mcw)$ is not $C^\infty$ and so one cannot expand it in Taylor series about $\{y=0\},$ as we did with $a_3(y)$ in the case $N=3.$   To justify this step one needs the following proposition which describes the behavior of  
 \begin{gather*}
 \mcf\left( (\p_u^{3} P)(y,\mcw(y)) V\right)(\xi, k_2\xi, k_3\xi), \text{ as } |\xi| \rightarrow \infty,
 \end{gather*}
 where as mentioned above  $\mcf$ denotes the Fourier transform.  This corresponds to the region $\eta_1=\xi,$ $\eta_2/\eta_1=\ka_2$ and $\eta_3/\eta_1=\ka_3.$  By symmetry, the same result holds if one switches the roles of $\eta_1,$ $\eta_2$ and $\eta_3.$
 
 \begin{prop}\label{asyFT}  Let $a(\eta)=a_1(\eta_1) a_2(\eta_2) a_3(\eta_3),$ with $a_j(\eta_j)\in S^{m}(\mr)$ with $m<-\fha.$  Let  $b(\eta)$ be such that for all $\eps>0,$
 \begin{gather}
 \lan \eta_1 \ran^{-m-\ha} \lan \eta_2 \ran^{-m-\ha} \lan \eta_3 \ran^{-m-\ha} \lan \eta\ran ^{1-\eps}b(\eta) \in L^2(\mr^3).\label{estb}
  \end{gather} 
 Then  for $\ka_2, \ka_3\in (\alpha,\beta),$ $\alpha>\alpha_0>0,$  the convolution $a\star b$ satisfies
 \begin{gather}
 a\star b(\xi, \ka_2\xi, \ka_3\xi)= a(\xi,\ka_2\xi,\ka_3\xi) \int_{\mr^3} b(\eta) d\eta + R(\ka_2,\ka_3,\xi),  \label{asyFT1}
 \end{gather}
  where $R(\ka_2,\ka_3,\xi)$ satisfies 
  \begin{gather}
   \begin{gathered}
\int_\alpha^\beta \int_\alpha^\beta \int_{\{|\xi| >1\} }|\xi|^{-6m -1+ r-\eps} |R(\ka_2,\ka_3,\xi)|^2 \; d\xi d \ka_2 d\ka_3 <\infty, \\ \text{ for all } \eps>0 
 \text{ provided }   r+\frac{2}{-m-\ha} \leq 1.
  \end{gathered} \label{asyFT2}
    \end{gather}
  \end{prop}
  One should remark that \eqref{asyFT2} is equivalent to saying that in the region 
  \begin{gather}
  \begin{gathered}
  \Gamma=\left\{ \; \frac{1}{C}\leq  |\eta_2/\eta_1|\leq C, \; \frac{1}{C} \leq |\eta_3/\eta_1|\leq C\right\} \text{ where projective coordinates } \\
  \eta_1=\xi, \;\ \ka_2=\eta_2/\eta_1, \;\  \ka_3=\eta_3/\eta_1 \text{ hold }, \\
   \int_{\Gamma}\left| |\eta|^{-3m-\tha+r/2} R(\eta)\right|^2 d\eta<\infty,
 \text{ provided }   r+\frac{2}{-m-\ha} < 1,
  \end{gathered} \label{NEWR}
    \end{gather}
  where the Jacobian of the change of variables gives the extra two powers of $\xi$ in \eqref{asyFT2}.
  \begin{proof}  By definition,
\begin{gather}
\begin{gathered}
a\star b(\xi,\ka_2\xi, \ka_3 \xi)= \int_{\mr^3} a(\xi\theta-\eta) b(\eta) \; d\eta, \text{ where } \\
 \theta=(1,\ka_2,\ka_3) \text{ and } (\xi-\eta_1, \ka_2 \xi-\eta_2, \ka_3\xi-\eta_3)= (\xi\theta-\eta),
\end{gathered} \label{astarb}
\end{gather}
and our goal is to show that
\begin{gather}
R(\ka_2,\ka_3,\xi)= \int_{\mr^3} \left( a(\xi\theta-\eta)-a(\xi\theta)\right) b(\eta) \; d\eta \label{astarb1}
\end{gather}
satisfies \eqref{asyFT2}.  We  pick $\del$ such that
\begin{gather}
0<  \del < \min(1,\alpha_0), \text{ where }  \ka_2, \ka_3>\alpha_0, \label{delint}
\end{gather}
and we split the region of integration of \eqref{astarb1} into several parts parts corresponding to whether
\begin{gather}
\begin{gathered}
|\eta_1-\xi|\geq \del\xi\text{ or } |\eta_1-\xi|\leq  \del |\xi|, \;\
 |\eta_2-\ka_2\xi|\geq \del |\xi| \text{ or }  |\eta_2-\ka_2\xi|\leq \del |\xi|, \\  |\eta_3-\ka_3\xi|\geq \del |\xi| \text{ or } 
 |\eta_3-\ka_3\xi|\leq \del |\xi|,
\end{gathered}\label{regions}
\end{gather}
We write
\begin{gather}
\begin{gathered}
R(\xi,\ka_2\xi,\ka_3\xi)= \int_{A_{GGG}(\xi,\eta)} (a(\xi\theta-\eta)-a(\xi\theta)) b(\eta) d\eta + 
\int_{A_{GGG}^c(\xi,\eta)} (a(\xi\theta-\eta)-a(\xi\theta)) b(\eta) d\eta =\\
\int_{A_{GGG}(\xi,\eta)} (a(\xi\theta-\eta)-a(\xi\theta)) b(\eta) d\eta + \int_{A_{GGG}^c(\xi,\eta)} a(\xi\theta-\eta) b(\eta) d\eta -a(\xi\theta) \int_{A_{GGG}^c(\xi,\eta)} b(\eta) d\eta= \\ \int_{A_{GGG}(\xi,\eta)} (a(\xi\theta-\eta)-a(\xi\theta)) b(\eta) d\eta + 
\sum_\bullet \int_{A_\bullet(\xi,\eta)} a(\xi\theta-\eta) b(\eta) d\eta- a(\xi\theta) \int_{A_{GGG}^c(\xi,\eta)} b(\eta) d\eta,
\end{gathered}\label{split-conv-0}
\end{gather}
where  $\bullet$ represents one of the possible indices below, 
\begin{gather}
\begin{gathered}
A_{GGG}(\xi,\eta)=\{ |\eta_1-\xi|\geq \del |\xi|,|\eta_2-\ka_2\xi|\geq\del|\xi|, |\eta_3-\ka_3\xi | \geq\del|\xi|\}, \\
 A_{GGL}(\xi,\eta)=\{|\eta_1-\xi|\geq \del |\xi|,|\eta_2-\ka_2\xi|\geq \del |\xi|, |\eta_3-\ka_3\xi|\leq \del |\xi|\}, \\
 A_{GLL}(\xi,\eta)=\{|\eta_1-\xi|\geq \del |\xi|,|\eta_2-\ka_2\xi|\leq \del |\xi|, |\eta_3-\ka_3\xi|\leq \del |\xi|\}, \\
  A_{GLG}(\xi,\eta)=  \{|\eta_1-\xi|\geq \del |\xi|,|\eta_2-\ka_2\xi|\leq \del |\xi|, |\eta_3-\ka_3\xi|\geq \del |\xi|\}, \\
 A_{LGG}(\xi,\eta)=  \{|\eta_1-\xi|\leq \del |\xi|,|\eta_2-\ka_2\xi|\geq \del |\xi|, |\eta_3-\ka_3\xi|\geq \del |\xi|\}, \\
A_{LLL}(\xi,\eta)= \{|\eta_1-\xi|\leq \del |\xi|,|\eta_2-\ka_2\xi|\leq \del |\xi|, |\eta_3-\ka_3\xi|\leq \del |\xi|\}, \\
  A_{LLG}(\xi,\eta)=  \{|\eta_1-\xi|\leq \del |\xi|,|\eta_2-\ka_2\xi|\leq \del |\xi|, |\eta_3-\ka_3\xi|\geq \del |\xi|\}, \\
  A_{LGL}(\xi,\eta)=  \{|\eta_1-\xi|\leq \del |\xi|,|\eta_2-\ka_2\xi|\geq \del |\xi|, |\eta_3-\ka_3\xi|\leq \del |\xi|\},
  \end{gathered} \label{splitconv}
\end{gather}
and   $A_{GGG}^c(\xi,\eta)$ is the complement of $A_{GGG}(\xi,\eta).$ The indices $G$ and $L$ indicate the order in which the signs $\geq$ and $\leq$ are used. 

We claim that the following estimates hold:
\begin{gather}
\begin{gathered}
\left|a(\xi\theta)\int_{A_{GGG}^c(\xi,\eta)} b(\eta) d\eta\right| \leq C |\xi|^{3m-1},  \\
\left|\int_{A_{GGG}(\xi,\eta)} (a(\xi\theta-\eta)-a(\xi\theta)) b(\eta) d\eta\right| \leq C |\xi|^{3m-1}, 
\end{gathered} \label{fiveineq-0}
\end{gather}
Let 
\begin{gather}
I_\bullet(\ka_2,\ka_3,\xi)=  \int_{A_\bullet(\xi,\eta)} a(\xi\theta-\eta) b(\eta) \; d\eta.\label{defIL}
\end{gather}
 We will show that for any $\eps>0,$
 
 \begin{gather}
 \begin{gathered}
\int_\alpha^\beta \int_\alpha^\beta \int_{|\xi|>1} |\xi|^{-6m-1+r-\eps} |I_{\bullet}(\ka_2,\ka_3,\xi)|^2 d\xi d\ka_2d\ka_3\leq C,\ \text{ provided }\\ 
 \bullet \in \{GGL, GLG, LGG\}   \text{ (has exactly two } \geq \text{ signs) and } r+\frac{2}{-m-\ha} \leq 1, \\
 \int_\alpha^\beta \int_\alpha^\beta \int_{|\xi|>1} |\xi|^{-6m -\eps} |I_\bullet(\ka_2,\ka_3,\xi)|^2 d\xi d\ka_2 d\ka_3 <\infty, \  \text{ provided } \\ \bullet\in \{LLG, LGL, GLL \}  \text{ (has exactly two } \leq \text{ signs) and } \\
 \int_\alpha^\beta \int_\alpha^\beta \int_{|\xi|>1} |\xi|^{-6m+1-\eps} |I_{LLL}(\ka_2,\ka_3,\xi)|^2 d\xi d\ka_2d\ka_3 < \infty.
 \end{gathered} \label{fiveineq}
  \end{gather}
  These estimates imply \eqref{asyFT2}.   We start by proving the first inequality in \eqref{fiveineq-0}.
 \begin{lemma}\label{ineq1}   Let $A_{GGG}(\xi,\eta)$ be defined as in \eqref{splitconv} and let $A_{GGG}^c(\xi,\eta)$ be its complement,  then 
 \begin{gather*}
 \left|a(\xi\theta)\int_{A_{GGG}^c(\xi,\eta)} b(\eta) d\eta\right| \leq C |\xi|^{3m-1}.
 \end{gather*}
 \end{lemma}
 \begin{proof}
The complement of $A_{GGG}(\xi,\eta),$ can be divided in three regions according to whether $|\eta_1-\xi|<\del |\xi|$ or  $|\eta_2-\ka_2\xi|<\del |\xi|$ or $|\eta_3-\ka_3\xi|<\del |\xi|,$ and since
\begin{gather*}
|\eta_1|=|\xi+\eta_1-\xi|\geq |\xi| -|\eta_1-\xi|, \\
|\eta_2| = |\ka_2\xi +\eta_2-\ka_2 \xi| \geq \ka_2 |\xi|- |\eta_2-\ka_2\xi|, \\
|\eta_3| = |\ka_3\xi +\eta_3-\ka_3 \xi| \geq \ka_3 |\xi|- |\eta_3-\ka_3\xi|.
\end{gather*}
and $\del$ satisfy \eqref{delint},  we conclude that  $|\eta_j|\geq C |\xi|,$ for  at least one value of $j\in \{1,2,3\}$ in  $A_{GGG}^c(\xi,\eta).$  Also,  since $\ka_2>\del$ and $\ka_3>\del,$  it follows that $|a(\xi\theta)|=|a_1(\xi)a_2(\ka_2\xi)a_3(\ka_3\xi)|\leq C|\xi|^{3m},$ for $|\xi|>1.$
Since $m<-\fha,$  we  have, for $|\xi|>1,$ 
\begin{gather*}
\begin{gathered}
\left||\xi|^{-3m+1} a(\xi\theta) \int_{A_{GGG}^c(\xi,\eta)} b(\eta) d\eta\right|\leq  C\sum_{j=1}^3  \int_{\{ |\eta_j|>C|\xi|\} } \frac{|\xi|}{|\eta_j|} |\eta_j| |b(\eta)| d\eta \leq \\ 
C \sum_{j=1}^3  \int_{\{ |\eta_j|>C|\xi|\} }  |\eta_j| |b(\eta)| d\eta\leq C
 \int_{\mr^3} \lan\eta_1\ran \lan\eta_2\ran  \lan\eta_3\ran |b(\eta)| d\eta = \\
 C \int_{\mr^3} \lan\eta_1\ran^{m+\tha} \lan\eta_2\ran^{m+\tha} \lan\eta_3\ran^{m+\tha}  |b(\eta)|  \lan\eta_1\ran^{-m-\ha} \lan\eta_2\ran^{-m-\ha} \lan\eta_3\ran^{-m-\ha}d\eta \leq \\
||b|| \left( \int_{\mr^3} \left[\lan\eta_1\ran^{m+\tha} \lan\eta_2\ran^{m+\tha} \lan\eta_3\ran^{m+\tha} \right]^2 d\eta\right)^\ha \leq C ||b||, \\
\text{ where }   ||b||^2= \int_{\mr^3} \left[\lan\eta_1\ran^{-m-\ha} \lan\eta_2\ran^{-m-\ha} \lan\eta_3\ran^{-m-\ha} |b(\eta)| \right]^2 d\eta.
\end{gathered}
\end{gather*}
This proves the Lemma.
\end{proof}
Now we prove the second estimate in \eqref{fiveineq-0}.
\begin{lemma}\label{ineq2} Let $A_{GGG}(\xi,\eta)$ be as defined in \eqref{splitconv}, then for $|\xi|>1,$
\begin{gather*}
\left|\int_{A_{GGG}(\xi,\eta)}\left( a(\xi\theta-\eta)-a(\xi\theta)) b(\eta)\right) d\eta\right|\leq C |\xi|^{3m-1}.
\end{gather*}
\end{lemma}
\begin{proof} We will split  this integral  into  two parts: 
\begin{gather*}
\int_{A_{GGG}(\xi,\eta)}\left( a(\xi\theta-\eta)-a(\xi\theta)\right) b(\eta) d\eta=  \int_{Z(\xi,\eta)\cap A_{GGG}(\xi,\eta)} \left( a(\xi\theta-\eta)-a(\xi\theta)\right)  b(\eta) d\eta+ \\
 \int_{A_{GGG}(\xi,\eta)\setminus Z(\xi,\eta)} \left( a(\xi\theta-\eta)-a(\xi\theta)\right)  b(\eta) d\eta, \\
 \text{ where }   Z(\xi,\eta)=\{|\eta_1|<(1-\del) |\xi|,|\eta_2|< (\ka_2-\del) |\xi|, |\eta_3|<(\ka_3-\del)|\xi|\}.
 \end{gather*}
We will show that
\begin{gather}
\begin{gathered}
|\int_{Z(\xi,\eta)\cap A_{GGG}(\xi,\eta)} \left( a(\xi\theta-\eta)-a(\xi\theta)\right)  b(\eta) d\eta|\leq C |\xi|^{3m-1}, \\
|a(\xi\theta)|\int_{A_{GGG}(\xi,\eta)\setminus Z(\xi,\eta)} |b(\eta)| d\eta \leq C |\xi|^{3m-1}, \\
|\int_{A_{GGG}(\xi,\eta)\setminus Z(\xi,\eta)} a(\xi\theta-\eta) b(\eta) \ d\eta| \leq  C |\xi|^{3m-1}.
\end{gathered}\label{REM}
\end{gather}

   Let
\begin{gather*}
W(\ka_2,\ka_3,\xi)=   \int_{Z(\xi,\eta)\cap A_{GGG}(\xi,\eta)} \left(a(\xi\theta-\eta)-a(\xi\theta) \right) b(\eta) d\eta.
 \end{gather*}
 Then
 \begin{gather*}
 W(\ka_2,\ka_3,\xi)= 
 \int_{Z(\xi,\eta)\cap A_{GGG}(\xi,\eta)}\left( \int_0^1 \frac{d}{dt} a(\xi\theta-t\eta) dt\right) b(\eta) d\eta=\\
 \int_0^1 \int_{Z(\xi,\eta)\cap A_{GGG}(\xi,\eta)} (\nabla a)(\xi\theta-t\eta)\cdot \eta \ b(\eta)\ d\eta dt.
 \end{gather*}
 But in $Z(\xi,\eta)$ we have
 \begin{gather*}
 |\xi-t\eta_1|\geq |\xi|-|t\eta_1| \geq |\xi| -t(1-\del)|\xi|\geq \del|\xi|, \\
  |\ka_2\xi-t\eta_2|\geq \ka_2 |\xi|-|t\eta_2| \geq \ka_2 |\xi| -t(\ka_2-\del)|\xi|\geq \del|\xi|, \\
  |\ka_3 \xi-t\eta_3|\geq \ka_3 |\xi|-|t\eta_3| \geq |\xi| -t(\ka_3-\del)|\xi|\geq \del|\xi|.
  \end{gather*}
Since $a_j\in S^m(\mr),$ it follows that, for $|\xi|>1,$
\begin{gather*}
|\nabla a(\xi\theta-t\eta)| \leq |a_1'(\xi-t\eta_1) a_2(\ka_2\xi- t\eta_2) a_3(\ka_3\xi-t\eta_3)|+ |a_1(\xi-t\eta_1) a_2'(\ka_2\xi-t\eta_2) a_3(\ka_3\xi-t\eta_3)|+ \\ |a_1(\xi-t\eta_1) a_2(\ka_2\xi-t\eta_2) a_3'(\ka_3\xi-t\eta_3)| \leq C |\xi|^{3m-1}.
\end{gather*}
Therefore
\begin{gather*}
|W(\ka_2,\ka_3,\xi)| \leq C |\xi|^{3m-1} \int_{\mr^3}( |\eta_1|+|\eta_2|+|\eta_3|) |b(\eta)| \ d\eta,
\end{gather*}
and arguing as in the proof of Lemma \ref{ineq1}, we conclude that  $|W(\ka_2,\ka_3,\xi)|\leq C|\xi|^{3m-1}.$ This ends the proof of the first inequality of \eqref{REM}.

 We  analyze $\int_{A_{GGG}(\xi, \eta)\setminus Z(\xi,\eta)} a(\theta \xi)  b(\eta) \; d\eta.$  In this region at least one of the $|\eta_j|$ is bounded from below by $C |\xi|.$  Then, again using that $\ka_2>\del,$ $j=2,3,$  and hence $|a(\xi\theta)|\leq C|\xi|^{3m},$
  we write as above
 \begin{gather*}
|\xi^{-3m+1}a(\xi\theta)|\int_{A_{GGG}(\xi, \eta)\setminus Z(\xi,\eta)}  |b(\eta)| \; d\eta \leq C \sum_{j=1}^3 \int_{\{|\eta_j|> C|\xi|\}}  \frac{\xi}{|\eta_j|} |\eta_j| |b(\eta)| d\eta \leq \\
C \sum_{j=1}^3 \int_{\{|\eta_j|> C|\xi|\}}  |\eta_j| |b(\eta)| d\eta \leq C||b||,
\end{gather*}
which proves the second inequality of \eqref{REM}

Now let us consider the case $\int_{A_{GGG}(\xi,\eta)\setminus Z(\xi,\eta)} a(\xi\theta-\eta) b(\eta) d\eta.$  Again, we have $|\eta_j|\geq C|\xi|$ for at least one $j\in\{1,2,3\}.$  But since the point is in $A_{GGG}(\xi,\eta),$  we also have $|\xi-\eta_1|>\del|\xi|,$  $|\ka_2\xi-\eta_2|>\del|\xi|$ and $|\ka_3\xi-\eta_3|>\del|\xi|,$ and therefore 
$|a(\xi\theta-\eta)|\leq C|\xi|^{3m}.$ 
Hence we have 
\begin{gather*}
|\xi^{-3m+1} \int_{A_{GGG}(\xi,\eta)\setminus Z(\xi,\eta)} a(\xi\theta-\eta) b(\eta) d\eta| \leq   \sum_{j=1}^3 \int_{\{|\eta_j|>C|\xi|\}} \frac{\xi}{|\eta_j|} |\eta_j| |b(\eta)| d\eta,
\end{gather*}
which can be bounded as above. This proves the third inequality of  \eqref{REM} and the second inequality of \eqref{fiveineq-0}.
\end{proof}

Next we estimate $I_{LLL}(\ka_2,\ka_3,\xi).$  
\begin{lemma}\label{ineqi4}  Let   $A_{LLL}(\xi,\eta)$ is defined in \eqref{splitconv} and let $I_{LLL}(\ka_2,\ka_3,\xi)$ be defined in \eqref{defIL},  
then
\begin{gather*}
\int_\alpha^\beta\int_\alpha^\beta \int_{|\xi|>1} |\xi|^{-6m+1-\eps} |I_{LLL}(\ka_2,\ka_3,\xi)|^2 d\xi d\ka_2d\ka_3 < \infty.
\end{gather*}
\end{lemma}
\begin{proof} First we observe that in the region $A_{LLL}(\xi,\eta),$ 
\begin{gather*}
(1-\del) |\xi| \leq |\eta_1| \leq (1+\del) |\xi|, \;
(\ka_2-\del) |\xi| \leq |\eta_2| \leq (\ka_2+\del) |\xi|, \;
(\ka_3-\del) |\xi| \leq |\eta_3| \leq (\ka_3+\del) |\xi|,
\end{gather*}
and so we conclude that 
\begin{gather*}
|\xi|^{-3m+\ha-\eps} |I_{LLL}(\ka_2,\ka_3,\xi)|\leq C |\xi| \int_{A_{LLL}(\xi,\eta)}  |a(\xi\theta-\eta)|  |B(\eta)| \; d\eta, \text{ where } \\
B(\eta)= \lan \eta_1\ran^{-m-\ha} \lan \eta_2\ran^{-m-\ha} \lan \eta_3\ran^{-m-\ha} \lan \eta\ran^{1-\eps}|b(\eta)|.
\end{gather*}
Therefore 
\begin{gather*}
\int_\alpha^\beta \int_\alpha^\beta  \int_{|\xi|>1} |\xi|^{-6m+1-\eps} |I_{LLL}(\ka_2,\ka_3,\xi)|^2 \ d\xi d\ka_2 d\ka_3 \leq \\
C \int_\alpha^\beta \int_\alpha^\beta  \int_{|\xi|>1} |\xi|^2\left|\int_{A_{LLL}(\xi,\eta)} |a(\xi\theta-\eta)|  |B(\eta)| \; d\eta \right|^2 d\xi d\ka_2 d\ka_3.
\end{gather*}
If we set $\xi=t_1,$ $t_2=\ka_2 \xi$ and  $t_3=\ka_3\xi,$ it follows from Young's inequality that
\begin{gather*}
\int_\alpha^\beta \int_\alpha^\beta |\xi|^2\left| \int_{A_{LLL}(\xi,\eta)} |a(\xi\theta-\eta)|  |B(\eta)| \; d\eta \right|^2 d\xi d\ka_2 d\ka_3\leq \\
\int_{\mr^3} \left| \int_{\mr^3} |a(t-\eta)|  |B(\eta)| \; d\eta \right|^2 dt\leq ||a||_{L^1(\mr^3)}^2||B||_{L^2(\mr^3)}^2.
\end{gather*}
This proves Lemma \ref{ineqi4}.
\end{proof}

\begin{lemma} \label{ineqi3}  Let $I_\bullet(\ka_2,\ka_3,\xi)$ be defined in \eqref{defIL}, with  $\bullet\in \{GLL,LLG,LGL\},$ then
\begin{gather*}
 \int_\alpha^\beta\int_\alpha^\beta\int_{|\xi|>1}  |\xi|^{-6m -2\eps} |I_\bullet(\ka_2,\ka_3,\xi)|^2 d\xi d\ka_2 d\ka_3 <\infty..
\end{gather*}
\end{lemma}
\begin{proof}  These domains $A_\bullet(\xi,\eta)$ have two $\leq$ signs and one $\geq$ sign. The argument used to estimate $I_{GLL}(\ka_2,\ka_3,\xi)$ can be used to estimate $I_{LLG}(\ka_2,\ka_3,\xi)$ and $I_{LGL}(\ka_2,\ka_3,\xi).$

On the domain $A_{GLL}(\xi,\eta)$ we have that $|\xi-\eta_1|\leq \del|\xi|,$ $(\ka_2-\del)|\xi| \leq |\eta_2| \leq  (\ka_2+\del) |\xi|$ and $(\ka_3-\del)|\xi| \leq |\eta_3| \leq (\ka_3+\del) |\xi|,$ and so we find that
\begin{gather*}
|\xi|^{-3m-\eps} |I_{GLL}(\ka_2,\ka_3,\xi)| \leq |\xi| \int_{A_{GLL}(\xi,\eta)} |\xi|^{-m-1-\eps/2}|a(\xi\theta-\eta)| B_1(\eta)  d\eta, \\
\text{ where } B_1(\eta)=  \lan \eta_2\ran^{-m-\ha} \lan \eta_3\ran^{-m-\ha} \lan \eta\ran^{1-\eps/2} |b(\eta)|
\end{gather*}
Since $-m-1>0,$ and $|\eta_1-\xi|\geq \del |\xi|$ on $A_{GLL}(\xi,\eta),$ we have that
\begin{gather*}
|\xi|^{-m-1-\eps/2} |a(\xi\theta-\eta)|\leq C |\eta_1-\xi_1|^{-m-1-\eps_2} |a_1(\xi-\eta_1)| |a_2(\ka_2\xi-\eta_2)| |a_3(\ka_3\xi-\eta_3)|\leq \\
C |\eta_1-\xi_1|^{-1-\eps/2} |a_2(\ka_2\xi-\eta_2)| |a_3(\ka_3\xi-\eta_3)|.
\end{gather*}
So we conclude that
\begin{gather*}
\int_\alpha^\beta\int_\alpha^\beta\int_{|\xi|>1}   |\xi|^{-6m -2\eps} |I_{GLL}(\ka_2,\ka_3,\xi)|^2 \ d\xi d\ka_2 d\ka_3\leq \\
\int_\alpha^\beta\int_\alpha^\beta\int_{|\xi|>1} |\xi|^2 \left|\int_{A_{GLL}(\xi,\eta)} |h(\xi\theta-\eta)| B_1(\eta) \ d\eta\right|^2 d\xi d\ka_2 d\ka_3,\\
\text{ where }  h(\xi\theta-\eta)= |\eta_1-\xi_1|^{-1-\eps/2} |a_2(\ka_2\xi-\eta_2)| |a_3(\ka_3\xi-\eta_3)|.
\end{gather*}

As above we set $t_1=\xi,$  $t_2=\ka_2 \xi$ and $t_3=\ka_3 \xi,$  and use Young's inequality to  deduce that
\begin{gather*}
\int_\alpha^\beta\int_\alpha^\beta\int_{|\xi|>1}   |\xi|^{-6m -2\eps} |I_j(\ka_2,\ka_3,\xi)|^2 \ d\xi d\ka_2 d\ka_3\leq 
C\int_{\mr^3} \left|\int_{\mr^3} |h(t-\eta)| |B_1(\eta)| \ d\eta\right|^2 dt \leq \\ C ||h||_{L^1(\mr^3)}^2||B_1||_{L^2(\mr^3)}^2.
\end{gather*}

This proves Lemma \ref{ineqi3}.
\end{proof}

\begin{lemma} \label{ineqi2}  Let $I_\bullet(\ka_2,\ka_3,\xi)$ be as in \eqref{defIL}, with $\bullet\in\{GGL,GLG,LGG\},$   then
\begin{gather*}
\int \xi^{-6m-1+r-6 \eps} |I_{\bullet}(\ka_2,\ka_3,\xi)|^2 d\xi d\ka_2d\ka_3\leq C.
\end{gather*}
\end{lemma}
\begin{proof} According to \eqref{splitconv} the region $A_{GGL}(\xi,\eta),$ is characterized by two $\geq$ signs and one $\leq$ sign and so the argument we use to estimate $I_{GGL}(\ka_2,\ka_3,\xi)$ also applies to estimate $I_{GLG}(\ka_2,\ka_3,\xi)$ and $I_{LGG}(\ka_2,\ka_3,\xi),$ which also have two $\geq$ signs and one $\leq$ sign.

We further divide $A_{GGL}(\xi,\eta)$ in two regions:
\begin{gather*}
A_{GGL}(\xi,\eta)=E_1(\xi,\eta) \cup E_2(\xi,\eta) \text{ where } \\ E_1(\xi,\eta)= A_{GGL}(\xi,\eta) \cap\left( \{ |\eta_1| > |\xi|^{\frac{1}{-m-\ha}} \}  \cup \{ |\eta_2| > |\xi|^{\frac{1}{-m-\ha}} \} \right)
\text{ and } \\
E_2(\xi,\eta)= A_{GGL}(\xi,\eta) \cap \{ |\eta_1| \leq  |\xi|^{\frac{1}{-m-\ha}}, \;\ |\eta_2| \leq  |\xi|^{\frac{2}{-m-\ha}}\}
\end{gather*}
and denote
\begin{gather*}
I_{E1}(\xi,\ka_2,\ka_3)= \int_{E_1(\xi,\eta)} a(\xi\theta-\eta) b(\eta) \ d\eta, \\
I_{E_2}(\xi,\ka_2,\ka_3)= \int_{ E_2(\xi,\eta)} a(\xi\theta-\eta) b(\eta) \ d\eta,
\end{gather*}
We will prove the following estimates:
\begin{gather}
\begin{gathered}
\int_\alpha^\beta\int_\alpha^\beta \int_{|\xi|>1} |\xi|^{-6m+1-6\eps} |I_{E_1}(\xi,\ka_2,\ka_3)|^2 \ d\xi d\ka_2 d\ka_3 <C, \\
\int_\alpha^\beta\int_\alpha^\beta \int_{|\xi|>1} |\xi|^{-6m-1-6\eps+r} |I_{E_2}(\xi,\ka_2,\ka_3)|^2 \ d\xi d\ka_2 d\ka_3 <C, \text{ provided } 0<r \leq 1-\frac{1}{-m-\ha}.
\end{gathered}\label{last-two}
\end{gather}

The argument used in the proof of previous lemmas still works to prove the bound for $I_{E_1}.$ We write
\begin{gather*}
 |\xi|^{-3m+\ha-3\eps} |I_{E_1}(\xi,\ka_2,\ka_3)|=  |\xi| \int_{ E_1(\xi,\eta)} |\xi|^{-2m-2-2\eps} |a(\xi\theta-\eta)| |\xi|^{-m+\tha-\eps} |b(\eta)| \ d\eta.
 \end{gather*}
Recall that on $E_1(\xi\eta)$ we have 
\begin{gather}
\begin{gathered}
|\xi-\eta_1|\geq \del|\xi|, \; |\ka_2\xi-\eta_2|\geq \del|\xi| \text{ and } (\ka_3-\del) |\xi| < |\eta_3| < (\ka_3+\del)|\xi|, \\
|\xi|\leq |\eta_1|^{-m-\ha} \text{ or } |\xi|\leq |\eta_2|^{-m-\ha},
\end{gathered}\label{aux-a2}
\end{gather}
 and so it follows that on $E_1(\xi,\eta),$
 \begin{gather*}
 |\xi|^{-2m-2-2\eps} |a(\xi\theta-\eta)| \leq |\xi-\eta_1|^{-m-1-\eps} |a_1(\xi-\eta_1)|  |\ka_2\xi-\eta_2|^{-m-1-\eps} |a_2(\ka_2\xi-\eta_2)| |a_3(\ka_3\xi-\eta_3)|\leq \\
 C |\xi-\eta_1|^{-1-\eps}   |\ka_2\xi-\eta_2|^{-1-\eps} |a_3(\ka_3\xi-\eta_3)|,\\
 |\xi|^{-m+\tha-\eps} |b(\eta)|\leq \lan \eta_3\ran^{-m-\ha}( \lan \eta_1\ran^{-m-\ha}+ \lan \eta_2\ran^{-m-\ha}) \lan \eta \ran^{1-\eps} |b(\eta)|,
 \end{gather*}
 and we conclude that
 \begin{gather*}
 |\xi|^{-3m+\ha-3\eps} |I_{E_1}(\xi,\ka_2,\ka_3)|\leq   C |\xi| \int_{E_1(\xi,\eta)} h(\xi\theta-\eta) B_2(\eta) \ d\eta, \\ \text{ where } 
 B_2(\eta)= \lan \eta_3\ran^{-m-\ha}( \lan \eta_1\ran^{-m-\ha}+ \lan \eta_2\ran^{-m-\ha}) \lan \eta \ran^{1-\eps} |b(\eta)| , \\
 \text{ and } h(\xi\theta-\eta)= |\xi-\eta_1|^{-1-\eps}   |\ka_2\xi-\eta_2|^{-1-\eps} |a_3(\ka_3\xi-\eta_3)|.
 \end{gather*}

  Therefore we have
 \begin{gather*}
\int_\alpha^\beta\int_\alpha^\beta \int_{|\xi|>1}  |\xi|^{-6m+1-6\eps}|I_{E_1}(\ka_2,\ka_3,\xi)| \ d\xi d\ka_2 d\ka_3 \leq  \\
\int_\alpha^\beta \int_\alpha^\beta \int_{|\xi|>1} |\xi|^2 \left| \int_{E_1(\xi,\eta)}  h(\xi\theta-\eta) B_2(\eta) \ d\eta\right|^2 \ d\xi d\ka_2 d\ka_3.
\end{gather*}
As in the proof of previous lemmas, we set $t_1=\xi,$ $t_2=\ka_2\xi$ and $t_3=\ka_3\xi,$ and as a consequence of Young's inequality we obtain,
\begin{gather*}
\int_\alpha^\beta\int_\alpha^\beta \int_{|\xi|>1}  |\xi|^{-6m+1-6\eps}|I_{E_1}(\ka_2,\ka_3,\xi)| \ d\xi d\ka_2 d\ka_3 \leq 
C \int_{\mr^3} \left| \int_{\mr^3} h(\xi\theta-\eta) B_2(\eta) \ d\eta\right|^2 dt \\ \leq C ||h||_{L^1(\mr^3)} ||B_2||_{L^2(\mr^3)}.
\end{gather*}
This proves the first inequality in \eqref{last-two}.

The second integral of \eqref{last-two} satisfies
\begin{gather*}
|I_{E_2}(\xi, \ka_2,\ka_3) | \leq \int_{E_2(\xi,\eta)} |a(\xi\theta-\eta)||b(\eta)|\ d\eta=\\
\int_{E_2(\xi,\eta)} |a(\xi\theta-\eta)\lan \eta_3\ran^{m+\ha}\lan \eta\ran^{-1+\eps} |B_3(\eta)|\ d\eta, \\
\text{ where } B_3(\eta)=  \lan \eta_3\ran^{-m-\ha}\lan \eta\ran^{1-\eps} |b(\eta)|.
\end{gather*}
Using the estimates for $|\eta_3|$ in \eqref{aux-a2} and the the Cauchy-Schwarz inequality  we obtain
\begin{gather}
\begin{gathered}
|I_{E_2}(\xi, \ka_2,\ka_3) |\leq C  |\xi|^{m-\ha+\eps} \int_{E_2(\xi,\eta)} |a(\xi\theta-\eta)||B_3(\eta)|\ d\eta \leq \\
C |\xi|^{m-\ha+\eps} \left[\int_{E_2(\xi,\eta)} |a(\xi\theta-\eta)|^2|\ d\eta\right]^\ha.
\end{gathered}\label{est-i22}
\end{gather}
Now we use that in $E_2(\xi,\eta),$ $|\eta_1|\leq |\xi|^{^\frac{1}{-m-\ha}}$ and $|\eta_2|\leq |\xi|^{^\frac{1}{-m-\ha}}.$ If we set $t_1=\xi-\eta_1,$ $t_2=\ka_2\xi-\eta_2$ and $t_3=\ka_3\xi-\eta_3,$ and we use $\mu=\frac{1}{-m-\ha}$ to simplify the notation, we have that
\begin{gather*}
\int_{E_2(\xi,\eta)} |a(\xi\theta-\eta)|^2\ d\eta\leq  
\left(\int_{\xi-|\xi|^\mu}^{\xi+|\xi|^\mu}|a_1(t_1)|^2 dt_1\right) \left(\int_{\ka_2\xi-|\xi|^\mu}^{\ka_2\xi+|\xi|^\mu}|a_2(t_1)|^2 dt_1\right) \int_\mr|a_3(t_3)|^2 dt.
\end{gather*}
Since $m<-\fha,$ then for $|\xi|$ large enough $|\xi-|\xi|^\mu| \geq C|\xi|$ and $|\ka_2\xi-|\xi|^\mu| \geq C|\xi|$ and so in this region $|a_1(t_1)|\leq C |\xi|^m$ and
$|a_2(t_2)|\leq C |\xi|^m,$ and so we obtain, for $|\xi|$ large enough,
\begin{gather*}
\int_{E_2(\xi,\eta)} |a(\xi\theta-\eta)|^2\ d\eta\leq C |\xi|^{4m+2\mu}.
\end{gather*}
So it follows from \eqref{est-i22} that, for $|\xi|>1,$ and $\eps>0,$
\begin{gather*}
|I_{E_2}(\xi, \ka_2,\ka_3) |\leq C |\xi|^{3m-\ha+\mu+\eps}.
\end{gather*}
Therefore,
\begin{gather*}
\int_\alpha^\beta\int_{\alpha}^\beta \int_{|\xi|>1} |\xi|^{-6m -1+r-2-\eps}| I_{E_2}(\xi, \ka_2,\ka_3) |^2 \ d\xi d\ka_2 d\ka_3\leq
C \int_{|\xi|>1} |\xi|^{r-2+2\mu-\eps } \ d\xi< \infty, \\ \text{ provided } r-1+2\mu\leq 0.
\end{gather*}

This proves the second inequality of \eqref{last-two} and ends the proof of Lemma \ref{ineqi2}.
\end{proof}

This also ends the proof of  Proposition \ref{asyFT}
\end{proof}
Now we can finish the proof of Theorem \ref{triple1}.  We know from Lemma \ref{Reg-Terms} that
\begin{gather*}
\chi(u-v-W)- E_+\left( \chi (\p_u^{3} P)(y,\mcw) V\right)
 \in \mch^{1,m}(\Omega_2),
\end{gather*}
and since $v, W \in H^\infty$ in a small neighborhood of $\{0\}$ and microlocally in the region where $\lan \eta_j \ran \geqs \lan \eta\ran,$  this implies that, when $\eta_j$ is elliptic, $j=1,2,3,$
\begin{gather*}
u- E_+\left( \chi (\p_u^{3} P)(y,\mcw) V\right)\in H^{-3m +\ha-}.
\end{gather*}
We write
\begin{gather*}
u- (\p_u^3 P)(0,\mcw(0))E_+(V)= \\ u- E_+\left((\p_u^{3} P)(y,\mcw) V\right)- E_+\left[ \left( (\p_u^3 P)(0,\mcw(0))- (\p_u^{3} P)(y,\mcw)\right) V\right].
\end{gather*}
In view of \eqref{reguL} and \eqref{NEED},  $a(\eta)= \mcf(V)(\eta)$ and $b(\eta)=\mcf( (\p_u^{3} P)(y,\mcw(y)))$ satisfy the hypotheses of Proposition \ref{asyFT} and since 
$\displaystyle (\p_u^{3} P)(y,\mcw(y)))=\frac{1}{(2\pi)^3}\int_{\mr^3} b(\zeta) d\zeta,$ we find that
\begin{gather*}
\mcf((\p_u^{3} P)(y,\mcw(y)) V)(\eta)=\frac{1}{(2\pi)^3} a\star b(\eta)= a(\eta) \frac{1}{(2\pi)^3} \int_{\mr^3} b(\zeta) d\zeta + R(\eta)= \\ (\p_u^{3} P)(0,\mcw(0)) \mcf(V)(\eta) + R(\eta),
\end{gather*}
where $R(\eta)$ satisfies \eqref{NEWR}.
We  conclude that, near $\{0\}$ and microlocally in the region where $\lan \eta_j\ran \geqs \lan \eta\ran,$ $j=1,2,3,$
\begin{gather*}
\left(\p_u^3 P)(0,\mcw(0))-  (\p_u^{3} P)(y,\mcw)\right) V \in H^{-3m-\tha+r},  \text{ provided } r<1-\frac{2}{-m-\ha},
\end{gather*}
and therefore, in a neighborhood of $\{0\}$ and microlocally in the region where $\lan \eta_j\ran \geqs \lan \eta\ran,$ $j=1,2,3,$
\begin{gather*}
u-(\p_u^3 P)(0,\mcw(0))E_+(V)\in H^{-3m-\ha+r/2}, \text{ provided } r<1-\frac{2}{-m-\ha},
\end{gather*}
and since $\mcw(0)=u(0),$ this ends the proof of Theorem \ref{triple1}.
\end{proof}

Now we prove Theorem \ref{triple}. Since $v_j\in I^{m-\oq}(\Omega,\Sigma_j),$  it follows that 
$v_j\in I H^{-m-\ha-}(\Omega,\Sigma_j),$  for  $j=1,2,3.$ Hence  know from Theorem \ref{REG} that after the triple interaction and away from the hypersurfaces, $u\in I H_{\loc}^{-m-\ha-}(\Omega^+, \mcw_\mcq).$  Then it follows that, away from the hypersurfaces, 
$u \in I^{M}(\Omega^+,\mcq)$ for some $M.$  But  Theorem \ref{triple1}  gives the order of the top singularity of $u$ microlocally near the conormal bundle of $\mcq$ and away from $\Sigma_j,$ $j=1,2,3,$ and we have to find it precisely. Since in the region where $\lan \eta_j\ran \geqs \lan \eta\ran,$ $V$ is a conormal distribution to $\{0\}$ whose symbol is in $S^{3m}(\mr^3),$ it follows that $V\in I^{3m+\frac34}(\Omega,\{0\}).$   Now we need to appeal to the  calculus of paired Lagrangian distributions, see the paper by Greeleaf and Uhlmann \cite{GreUhl}. We know that $E_+$ is a paired Lagrangian distribution in the class $I^{p, l}(N^*\diag, \La)$ where $p = -3/2,$ and  $l = -1/2.$ Then we can apply Proposition 2.1 of \cite{GreUhl} to conclude that, away from $\{0\},$ $E_+(V)\in I^{3m-\frac34}(\Omega, \mcq).$  This concludes the proof of Theorem \ref{triple}.

\section{Acknowledgements}  
  S\'a Barreto  thanks the Simons Foundation for its support under grant (\#349507, Ant\^onio S\'a Barreto).  The authors are very grateful to an anonymous referee for carefully reading the paper and making several suggestions that helped improve it.


\end{document}